\documentclass[11pt]{amsart}

\usepackage[T1]{fontenc}
\usepackage[latin1]{inputenc}
\usepackage{amsmath,amssymb,amsthm}
\usepackage{enumitem}
\usepackage{amsrefs}
\usepackage{hyperref}

\usepackage{calc}
\makeatletter
\renewcommand{\tocsection}[3]{%
  \indentlabel{\ignorespaces#1 \makebox[\widthof{00.}][r]{#2\@ifnotempty{#2}{.}}\quad}#3}
\renewcommand{\tocsubsection}[3]{%
  \indentlabel{\@ifnotempty{#2}{\ignorespaces#1 \makebox[\widthof{00.0.}][r]{\small{#2.}}\,\,  }}\small{#3}}
\makeatother

\theoremstyle{plain}
\newtheorem{thm}{Theorem}[section]
\newtheorem{lem}[thm]{Lemma}
\newtheorem{prop}[thm]{Proposition}
\newtheorem{cor}[thm]{Corollary}

\theoremstyle{definition}
\newtheorem{defn}[thm]{Definition}

\newtheorem{rem}[thm]{Remark}
\newtheorem{exa}[thm]{Example}

\theoremstyle{remark}

\newcommand{\bB}{\mathbb B}
\newcommand{\bC}{\mathbb C}
\newcommand{\bD}{\mathbb D}

\newcommand{\bN}{\mathbb N}

\newcommand{\bR}{\mathbb R}

\newcommand{\bT}{\mathbb T}

\newcommand{\cA}{\mathcal A}

\newcommand{\cD}{\mathcal D}

\newcommand{\cH}{\mathcal H}
\newcommand{\cI}{\mathcal I}
\newcommand{\cJ}{\mathcal J}
\newcommand{\cK}{\mathcal K}
\newcommand{\cL}{\mathcal L}

\newcommand{\cT}{\mathcal T}
\newcommand{\cU}{\mathcal U}

\newcommand{\cX}{\mathcal X}

\newcommand{\cZ}{\mathcal Z}

\newcommand{\kW}{\mathfrak W}

\DeclareMathOperator{\ran}{ran}
\DeclareMathOperator{\dist}{dist}
\DeclareMathOperator{\Mult}{Mult}

\DeclareMathOperator{\TSH}{TS(\Mult(\cH))}
\DeclareMathOperator{\HENH}{Hen(\Mult(\cH))}
\DeclareMathOperator{\HEN}{Hen}
\DeclareMathOperator{\TS}{TS}
\DeclareMathOperator{\CB}{CB}

\renewcommand{\Re}{\operatorname{Re}}

\newcommand{\tfa}{\text{ for all }}

\newcommand{\qand}{\quad\text{and}\quad}

\newcommand{\ol}[1]{\overline{#1}}
\newcommand{\ep}{\varepsilon}
\newcommand{\dlim}{\lim\limits}

\author{Kenneth R. Davidson}
\address{Pure Mathematics Department, 
University of Waterloo,
Waterloo, ON\ N2L 3G1, 
Canada}
\email{krdavids@uwaterloo.ca}

\author{Michael Hartz}
\address{	Fachrichtung Mathematik\\
Universit\"at des Saarlandes\\
66123
Saarbr\"ucken, 
Germany}
\email{hartz@math.uni-sb.de}

\title[Interpolation and duality]{Interpolation and duality\\ in algebras of multipliers on the ball}


\begin{document}

\begin{abstract}
We study the multiplier algebras $A(\mathcal{H})$ obtained as the closure of the polynomials on certain reproducing kernel Hilbert spaces $\mathcal{H}$ on the ball $\mathbb{B}_d$ of $\mathbb{C}^d$.
Our results apply, in particular, to the Drury--Arveson space, the Dirichlet space and the Hardy space on the ball.
We first obtain a complete description of the dual and second dual spaces of $A(\mathcal H)$ in terms of the complementary bands of Henkin and totally singular measures for $\operatorname{Mult}(\mathcal{H})$.
This is applied to obtain several definitive results in interpolation.
In particular, we establish a sharp peak interpolation result for compact $\operatorname{Mult}(\mathcal{H})$-totally null sets
as well as a Pick and peak interpolation theorem. Conversely, we show that a mere interpolation set
is $\operatorname{Mult}(\mathcal{H})$-totally null.
\end{abstract}

\maketitle


\section{Introduction}

Classical peak interpolation is concerned with finding a disc algebra function that solves an interpolation
problem on the boundary of the unit disc $\mathbb{D}$ in the complex plane. In this setting, the Rudin--Carleson theorem
\cite{Carleson57,Rudin56} says that given any compact set $E \subset \partial \bD$ of
linear Lebesgue measure $0$ and any continuous function $h \in C(E)$ with $\|h\|_\infty \le 1$,
there exists $f \in A(\bD)$ with $f \big|_{E} = h$ and $\|f\|_\infty \le \|h\|_\infty$.
Moreover, if $h$ is not identically zero, one may achieve that $|f(z)| < \|h\|_\infty$ for $z \in \ol{\bD} \setminus E$, explaining
the term peak interpolation. For a discussion of the Rudin--Carleson theorem,
we  refer the reader to \cite[Chapter II]{Gamelin69}.
For the ball algebra $A(\bB_d)$ in dimension $d\ge2$, the notion of Lebesgue measure zero is replaced by the smallness property 
of being a null set with respect to every representing measure for the origin, known as a totally null set.
A general result of Bishop \cite{Bishop62} (see \cite[Chapter 10]{Rudin08}) shows that one can obtain 
peak interpolation on any compact totally null set.

A different interpolation problem for the disc algebra is the classical Pick interpolation problem.
It is about finding a function in the disc algebra that solves an 
interpolation problem with interpolation nodes in the open unit disc. 
The solution is given by Pick's theorem \cite{Pick15}, which can be stated as follows.
Given interpolation nodes $z_1,\ldots,z_n \in \bD$ and targets $\lambda_1,\ldots,\lambda_n \in \bC$,
there exists a function $f$ belonging to the disc algebra $A(\bD)$ with 
 $f(z_i) = \lambda_i$ for $1 \le i \le n$ and $\|f\|_{\infty} \le 1$ if and only if the Pick matrix
\begin{equation}
  \label{eqn:Pick_matrix}
  \Big[ \frac{1 - \lambda_i \overline{\lambda_j}}{1 - z_i \overline{z_j}} \Big]
\end{equation}
is positive semidefinite. 
Pick's theorem, and a subsequent reformulation due to Sarason \cite{Sarason67}, have had a profound influence on operator theory.
This result has been extended to a large class of reproducing kernel Hilbert spaces;
see \cite{AM02}.

Recently, Izzo \cite{Izzo18} studied the problem of simultaneous Pick and peak interpolation in the context of uniform algebras;
in particular, this includes the disc algebra.
Given
Pick and peak interpolation data that are solvable individually,
it is not always possible to find one function of norm $1$ that solves
both problems simultaneously.
Indeed, if the Pick matrix \eqref{eqn:Pick_matrix} is positive and singular, then the Pick interpolation problem
has a unique solution; and generally this will differ from the boundary datum.
Nevertheless, the problem can be solved with an arbitrarily small increase in norm.

We will establish sharp analogues of both the peak interpolation results and the simultaneous Pick and peak interpolation result
for certain algebras of multipliers of a large class of reproducing kernel Hilbert spaces  on the ball.
There are added complications due to the fact that the multiplier norm is larger than the supremum norm.
Duality methods are the key to controlling the multiplier norm of the interpolating functions.
The duality approach to interpolation theorems is classical; see for instance \cite[Theorem 5.9]{Rudin91}.
Roughly speaking, the idea is that an interpolation theorem asserts that a particular restriction mapping is surjective (or even a quotient mapping). 
By duality, this is equivalent to saying that the adjoint mapping is bounded below (or even an isometry).

In \cite{CD16b}, Clou\^atre and the first author established a functional calculus for absolutely continuous row contractions.
The arguments were based on a duality theory for Drury-Arveson space developed in \cite{CD16}.
In \cite{BHM17}, Bickel, M\textsuperscript{c}Carthy and the second author established the analogous functional calculus for a broad range of reproducing kernel Hilbert spaces on the ball.
Their methods avoided the use of duality for these spaces. However, in the course of our work on the present paper, we began by considering  the Pick and peak interpolation problem proposed in \cite{CD18}.
The solution for Drury-Arveson space again required duality. We decided to develop the duality theory for the larger class of spaces studied in \cite{BHM17}.

A reproducing kernel Hilbert space on the Euclidean open unit ball $\bB_d \subset \mathbb{C}^d$ is said to be a \emph{regular unitarily invariant space} 
if its reproducing kernel is of the form
\begin{equation*}
  K(z,w) = \sum_{n=0}^\infty a_n \langle z,w \rangle^n,
\end{equation*}
where $a_0 =1$, $a_n > 0$ for all $n \in \bN$ and $\dlim_{n \to \infty} \frac{a_n}{a_{n+1}} = 1$.
Throughout this paper, we will assume that $d < \infty$.
Examples are the Hardy space on the disc $H^2(\bD)$ and on the ball $H^2(\bB_d)$,
the Drury--Arveson space $H^2_d$ and the Dirichlet space.
More discussion and examples can be found in Subsection \ref{ss:RKHS}.
If, in addition, $\cH$ satisfies a version of Pick's interpolation theorem, then we call it a regular unitarily invariant Pick space. 
This class includes the Hardy space $H^2(\bD)$, the Dirichlet space and the Drury--Arveson space $H^2_d$,
but not the Hardy space $H^2(\bB_d)$ for $d \ge 2$. A precise definition will also be given in Subsection \ref{ss:RKHS}.
Given a regular unitarily invariant space $\cH$, the polynomials are multipliers of $\cH$, so we may define
\begin{equation*}
  A(\cH) = \overline{\bC[z_1,\ldots,z_d]}^{\|\cdot\|} \subset \Mult(\cH).
\end{equation*}
If $\cH = H^2(\bD)$, then $A(\cH)$ is the disc algebra. More generally,
if $\cH$ is the Hardy space on $\bB_d$, then $A(\cH) = A(\bB_d)$, the ball algebra. We remark, however, that in
many cases of interest, such as the Dirichlet space and the Drury--Arveson space $H^2_d$ for $d \ge 2$,
$A(\cH)$ is not a uniform algebra, and the multiplier norm is not comparable to the supremum norm.
Since the multiplier norm dominates the supremum norm,
\begin{equation*}
  A(\cH) \subset A(\bB_d) \cap \Mult(\cH) \subset A(\bB_d) \cap \cH,
\end{equation*}
but these inclusions are often strict.

\begin{defn}
Let $\cH$ be a regular unitarily invariant space on $\bB_d$.
\begin{enumerate}
  \item 
A regular Borel measure $\mu$ on $\partial \bB_d$ is called \emph{$\Mult(\cH)$-Henkin} if the integration functional
\begin{equation*}
  \bC[z_1,\ldots,z_d] \to \bC, \quad p \mapsto \int_{\partial \bB_d} p \, d \mu,
\end{equation*}
extends to a weak-$*$ continuous functional on $\Mult(\cH)$.
\item
A Borel set $E \subset \partial \bB_d$ is called \emph{$\Mult(\cH)$-totally null} if $|\mu|(E) = 0$ for every $\Mult(\cH)$-Henkin measure $\mu$.
\item
A regular Borel measure $\mu$ on $\bB_d$ is called \emph{$\Mult(\cH)$-totally singular}
if $\mu$ is singular with respect to every $\Mult(\cH)$-Henkin measure. 
\end{enumerate}
\end{defn}

Let $\TSH \subset M(\partial \bB_d)$ be the space of all $\Mult(\cH)$-totally singular measures.
If $\cH = H^2(\bD)$, then the F.\ and M.\ Riesz theorem implies that a measure is $H^\infty(\bD)$-Henkin
if and only if it is absolutely continuous with respect to Lebesgue measure.
Hence a Borel set is $H^\infty(\bD)$-totally null if and only if it is a Lebesgue null set.
The totally singular measures are just the measures singular to Lebesgue measure.

If $\cH$ is the Hardy space on $\bB_d$, then the corresponding Henkin measures are classical Henkin measures, and theorems
of Henkin and Cole--Range imply that a measure is $H^\infty(\bB_d)$-Henkin if and only if
it is absolutely continuous with respect to some representing measure of the origin \cite[Chapter 9]{Rudin08}.
The measures singular to all representing measures are the classical totally singular measures.
The Glicksberg--K\"onig--Seever theorem and the theorems of Henkin and Cole--Range show that the dual of the ball algebra is given by
\begin{equation*}
  A(\bB_d)^* = H^\infty(\bB_d)_* \oplus_1 \TS(H^\infty(\mathbb{B}_d)),
\end{equation*}
where $H^\infty(\bB_d)_*$ is the standard predual of $H^\infty(\bB_d)$
and $\oplus_1$ denotes the direct sum of two normed spaces, equipped with the sum of the norms;
see also Subsection \ref{ss:os} for background on direct sums.
This material can be found in \cite[Chapter 9]{Rudin08}.

Henkin measures and totally null sets for the Drury--Arveson space were studied by Clou\^atre and the first author in \cite{CD16},
also in the context of peak interpolation.
They showed that for the Drury--Arveson space,
\begin{equation*}
  A(H^2_d)^* = \Mult(H^2_d)_* \oplus_1 \TS(\Mult(H^2_d)).
\end{equation*}
This theorem was proved using results from the theory of free semigroup algebras.

Our description of the dual space of $A(\cH)$ contains these three results as special cases.

\begin{thm}
  \label{thm:A_H_decomp_intro}
  Let $\cH$ be a regular unitarily invariant space on $\bB_d$. Then
  \begin{equation*}
    A(\cH)^* = \Mult(\cH)_* \oplus_1 \TSH.
  \end{equation*}
\end{thm}

A refinement of this result will be proved in Theorem \ref{thm:A_H_decomp}.
Our proof is ultimately dilation theoretic and related to the methods in \cite{BHM17}.

This leads to a nice description of the second dual as well:

\begin{thm}
  \label{thm:A_H_double_dual_intro}
  Let $\cH$ be a regular unitarily invariant space on $\bB_d$. Then there is an abelian von Neumann algebra $\kW_s$ so that
  \begin{equation*}
    A(\cH)^{**} = \Mult(\cH) \oplus_\infty \kW_s.
  \end{equation*}
\end{thm}

Here, $\oplus_\infty$ denotes the direct sum of two normed spaces, equipped with the maximum
of the two norms; see also Subsection \ref{ss:os}.
A more refined version of this result, including a precise description of $\kW_s$, is found in Theorem~\ref{T:second_dual}.

We now turn to our interpolation results.

In the theory of uniform algebras, the notion of a peak interpolation set is stronger than being an interpolation set.
That is, given a function algebra $A \subset C(X)$, a closed subset $E \subset X$ is an interpolation set if for
every $h\in C(E)$, there is an element $f\in A$ so that $f|_E = h$. 
When this happens, the open mapping theorem yields some norm control. 
It is a peak interpolation set if in addition, one can arrange that $|f(x)| < \|h\|_\infty$ for all $x \in X \setminus E$,
provided that $h$ is not identically zero.
It is a peak set if there is a function $g\in A$ so that $g|_K=1$ and $|g(x)| < 1$ for all $x \in X \setminus E$.
It is routine in the uniform algebra context to show that a set which is both a peak set and an interpolation set is a peak interpolation set \cite[Lemma 20.1]{Stout71}.

In the special case of the ball algebra, Rudin \cite[Chapter 10]{Rudin08} explains that these three notions coincide
and are also equivalent to being totally null and to being the zero set of a function in $A(\bB_d)$.
In the case of Drury-Arveson space, it was shown  with considerable effort in \cite[Theorem 9.5]{CD16} that a
closed $\Mult(H^2_d)$-totally null set $E$ admits peak interpolation in weaker sense.
This provided the classical pointwise inequality but required $\|f\|_{\Mult(H^2_d)} \le (1+\ep) \|g\|_\infty$ for some $\ep>0$.
We obtain a sharper version of Bishop's theorem that compact $\Mult(\cH)$-totally null sets are peak interpolation sets in our setting.
The sharp norm control is obtained by using the theory of $M$-ideals.

\begin{thm}  \label{thm:Bishop_intro}
Let $\cH$ be a regular unitarily invariant space on $\bB_d$ and let $E \subset \partial \bB_d$ be compact and $\Mult(\cH)$-totally null. 
Let $g \in C(E)$ be not identically zero. 
Then there exists an $f \in A(\cH)$ with
 \begin{enumerate}[label=\normalfont{(\arabic*)}]
  \item $f \big|_E = g$,
  \item $|f(z)| < \|g\|_\infty$ for every $z \in \ol{\bB_d} \setminus E$, and
  \item $\|f\|_{\Mult(\cH)} =  \|g\|_\infty$.
 \end{enumerate}
\end{thm}

A slight improvement of this result, which also applies to matrices of multipliers,
will be proved in Theorem \ref{thm:Bishop}. In Theorem \ref{thm:Bishop_linear}, we will
show that there even exists a linear operator of peak interpolation, meaning $f$ can be chosen to depend linearly on $g$.

We mention that a somewhat different boundary interpolation result for Besov--Sobolev spaces on the ball was previously
obtained by Cohn and Verbitsky; see \cite[Theorem 3]{CV95} and also the references therein. They consider interpolation
in the larger space $\mathcal{H} \cap A(\mathbb{B}_d)$, but impose a capacitary condition on the interpolation set.
Peak interpolation theorems in a non-commutative setting were proved by Blecher, Hay and Read. 
See for instance \cite{Blecher13} for an overview. While these results are in a similar spirit to Theorem \ref{thm:Bishop_intro},
it seems that we need to develop our full theory here before we can fit Theorem \ref{thm:Bishop_intro} into their framework.

Before obtaining the sharp peak interpolation result, we first establish simultaneous Pick and peak interpolation.
The reason for doing this first is that for the special case of empty Pick component, one obtains important 
information that is a step towards the sharp peak interpolation result.

\begin{thm}
  \label{thm:pick_peak_intro}
    Let $\cH$ be a regular unitarily invariant Pick space on $\bB_d$ with kernel $K$.
    Let $F = \{z_1,\ldots,z_n \} \subset \bB_d$ and $\lambda_1,\ldots,\lambda_n \in \bC$ with
    \begin{equation*}
      \big[ K(z_i,z_j) ( 1 - \lambda_i \overline{\lambda_j}) \big] \ge 0.
    \end{equation*}
    Let $E \subset \partial \bB_d$ be compact and $\Mult(\cH)$-totally null, and let $h \in C(E)$ with $\|h\|_\infty \le 1$.
    Then for each $\varepsilon > 0$, there exists $f \in A(\cH)$ with
  \begin{enumerate}[label=\normalfont{(\arabic*)}]
    \item $f(z_i) = \lambda_i$ for $1 \le i \le n$,
    \item $f |_E = h$, and
    \item $\|f\|_{\Mult(\cH)} \le 1 + \varepsilon$.
  \end{enumerate}
\end{thm}

Our result shows that the restriction map to the set $F\cup E$ is a complete quotient map of $A(\cH)$ onto $\Mult(\cH)|_F \oplus_\infty C(E)$.
Theorem~\ref{thm:pick_peak_intro} will be obtained in Corollary \ref{cor:pick_peak}.
Taking $\cH = H^2(\bD)$, we recover Izzo's theorem in the case of the disc algebra. 
The case of the Drury--Arveson space and of a single point in $\bB_d$ (i.e.\ $n=1$) was established
by Clou\^atre and the first author in \cite[Corollary 3.8]{CD18}.
It turns out that the Pick property of $\cH$ is only needed to have a concise criterion for the solvability of the Pick problem. 
In Theorem \ref{thm:pick_peak_general}, we will provide a more general result that does not require the Pick property and for instance
also applies to the Hardy space on the ball.
Our proof is different from Izzo's proof, as $A(\cH)$ is typically not a uniform algebra in our setting.

In Section~\ref{S:ideals}, we establish a few results about ideals inspired by \cite{CD18}.
In particular, we provide an analogue of a theorem of Rudin and Carleson \cite{Carleson57, Rudin57} describing ideals of the disk algebra.
This was generalized in a somewhat less precise way for the ball algebra by Hedenmalm \cite{Hedenmalm89}
and by Clou\^atre and the first author \cite{CD18} for multipliers on Drury-Arveson space.
Our result is in the same spirit. We let $Z(\cJ)$ denote the set of common zeros of functions in $\cJ$.
If $E$ is a closed subset of $\overline{\mathbb{B}_d}$, let $\mathcal{I}(E) = \{f \in A(\mathcal{H}): f \big|_E = 0 \}$.

\begin{thm}\label{T:ideals_intro}
Let $\cH$ be a regular unitarily invariant space on $\bB_d$, and let $\cJ$ be a closed ideal in $A(\cH)$.
Let $E = Z(\cJ) \cap \partial\bB_d$, and let $\tilde\cJ$ be the weak-$*$ closure of $\cJ$ in $\Mult(\cH)$.
Then
\[ \cJ = \tilde\cJ \cap \cI(E) .\]
\end{thm}

After establishing our peak interpolation theorem, we investigate when there are non-empty $\Mult(\cH)$-totally null sets.
We show that either singleton sets on the boundary are not totally null, in which case all multipliers extend to be continuous on the closed ball,
or boundary points are totally null and there are interpolating sequences for $\Mult(\cH)$. 
When the kernel is bounded (i.e. $\sum_{n\ge0} a_n < \infty$), it is easy to see that the first case applies.
We construct an unbounded kernel with no totally null sets as well.

In Section~\ref{S:interpolation sets}, we establish a very strong converse to our various interpolation theorems.

\begin{thm}\label{T:interpolation set_intro}
Let $\cH$ be a regular unitarily invariant space and suppose that there exist non-empty $\Mult(\cH)$-totally null sets.
Let $E \subset \partial \bB_d$ be a compact set.
If the restriction map from $A(\cH)$ into $C(E)$ is surjective, then $E$ is $\Mult(\cH)$-totally null.
\end{thm}

We also show that if there are no non-empty $\Mult(\mathcal{H})$ totally null sets, then there are no infinite
compact interpolation sets; see Proposition \ref{prop:no_interpolation}.

In Section~\ref{S:zero_sets}, we answer a question from \cite{CD18} about zero sets of $A(H^2_d)$.
It is shown there that every closed totally null subset of $\partial \bB_d$ is the zero set of a function in $A(H^2_d)$.
This is also the case in our setting. 
It was asked whether the converse was true. 
The second author showed in \cite{Hartz17} that there are $\Mult(H^2_d)$-Henkin measures which are not Henkin measures in the classical sense. 
This is used to demonstrate that there are zero sets of $A(H^2_d)$ which are not $\Mult(H^2_d)$-totally null.

Finally in the last section, we show that the following properties coincide on these spaces.
We say that $E \subset \partial \bB_d$ is a peak set if there is a function $f\in A(\cH)$ such that $f|_E = 1 = \|f\|_{\Mult(\mathcal{H})}$ and $|f(z)|<1$ on $\ol{\bB_d}\setminus E$.

\begin{thm} \label{T:equivalence_intro}
Let $\cH$ be a regular unitarily invariant space on $\bB_d$, and let $E \subset \partial \bB_d$  be compact. 
The following are equivalent:
  \begin{enumerate}
    \item[\normalfont{(TN)}] $E$ is $\Mult(\cH)$-totally null.
    \item[\normalfont{(PI)}] $E$ is a peak interpolation set.
    \item[\normalfont{(P)}] $E$ is a peak set.
    \item[\normalfont{(PPI)}] $E$ is a Pick-peak interpolation set.
  \end{enumerate}
Moreover these properties imply the corresponding complete versions of {\normalfont{(PI)}} and {\normalfont{(PPI)}} for matrix valued functions.
Furthermore, if there exist non-empty $\Mult(\cH)$-totally null sets, then this is also equivalent to
  \begin{enumerate}
    \item[\normalfont{(I)}] $E$ is an interpolation set.
  \end{enumerate}
\end{thm}

\pagebreak[3]
\section{Preliminaries}

\subsection{Reproducing kernel Hilbert spaces on \texorpdfstring{$\bB_d$}{Bd}}
\label{ss:RKHS}

We refer the reader to the books \cite{AM02} and \cite{PR16} for background on reproducing kernel
Hilbert spaces.
Let $\cH$ be a reproducing kernel Hilbert space on $\bB_d$, where $d \in \bN$, and
let $K$ denote the reproducing kernel of $\cH$.
We say that $\cH$ is unitarily invariant if $K$ is of the form
\begin{equation*}
  K(z,w) = \sum_{n=0}^\infty a_n \langle z,w \rangle^n,
\end{equation*}
where $a_0 =1$ and $a_n > 0$ for all $n \in \bN$.
If in addition, $\dlim_{n \to \infty} \frac{a_n}{a_{n+1}} =1$,
then we call $\cH$ a regular unitarily invariant space.
We think of this condition as a regularity condition because
of the following principle.
If $\cH$ is a reproducing kernel Hilbert space on $\bB_d$, it is natural
to assume that the radius of convergence of the power series
$\sum_{n=0}^\infty a_n t^n$ is $1$. In this case, if the limit $\dlim_{n \to \infty} \frac{a_n}{a_{n+1}}$
exists, then it equals $1$.

The class of regular unitarily invariant spaces is a frequently studied class of Hilbert function
spaces on the ball. It includes in particular the classical Hardy space $H^2(\mathbb{D})$, 
the Dirichlet space, the Bergman space, their counterparts on the ball, as well
as the Drury--Arveson space $H^2_d$, which plays a key role in multivariable operator theory \cite{Arveson98,Drury78}. More generally, for each
$a \in (0,\infty)$, the reproducing kernel Hilbert space $\mathcal{K}_a(\mathbb{B}_d)$ with kernel
\begin{equation*}
  \frac{1}{(1 - \langle z,w \rangle)^a}
\end{equation*}
belongs to this class. The Drury--Arveson space is obtained at $a=1$, the Hardy space at $a=d$
and the Bergman space at $a=d+1$. Closely related are the spaces $\cH_s(\bB_d)$ with reproducing kernel
\begin{equation*}
  K(z,w) = \sum_{n=0}^\infty (n+1)^s \langle z, w \rangle^n
\end{equation*}
for $s \in \bR$. Here, $s=0$ corresponds to the Drury--Arveson space and $s=-1$
yields the Dirichlet space. It is known that if $s = a - 1 > -1$, then $\mathcal{H}_s(\mathbb{B}_d) = \mathcal{K}_a(\mathbb{B}_d)$ with equivalence of norms; this can be seen by expanding the kernel of $\mathcal{K}_a(\mathbb{B}_d)$
into a power series.

Let $\cH$ be a regular unitarily invariant Hilbert space with multiplier algebra $\Mult(\cH)$.
Then the coordinate functions are multipliers of $\cH$. 
Let $A(\cH)$ denote the norm closure of the polynomials inside of $\Mult(\cH)$.
If $\cH$ is the Hardy space on the disc or on the ball, then $A(\cH)$ is the disc algebra
or the ball algebra, respectively. If $\cH$ is the Drury--Arveson space, then $A(\cH)$
is Arveson's algebra $\cA_d$, see \cite{Arveson98}.

The monomials $z^k = z_1^{k_1}\dots z_d^{k_d}$ for $k \in \bN_0^d$ form an orthogonal basis for $\cH$,
and thus one can show that elements of $\cH$ are holomorphic functions on $\bB_d$ \cite[Proposition~4.1]{GHX04}.
Since the multiplier norm dominates the supremum norm, $A(\cH)$
is contained in the ball algebra $A(\bB_d)$. In particular,
every function in $A(\cH)$ extends uniquely to a continuous function on $\ol{\bB_d}$.
Conversely, $A(\cH)$ contains every function that is holomorphic in a neighborhood of $\ol{\bB_d}$. For instance,
this can be seen from the fact that the Taylor spectrum of the tuple
$(M_{z_1},\ldots,M_{z_d})$ on $\cH$ is equal to $\ol{\bB_d}$ (see \cite[Theorem 4.5]{GHX04}) by an application
of the Taylor functional calculus; see \cite{Mueller07}.
Indeed, $(M_{z_1}, \ldots,M_{z_d})$ is an essentially normal $d$-variable weighted shift.
Moreover, we obtain an exact sequence \cite[Theorem~4.6]{GHX04}
\begin{equation}
  \label{eqn:short_exact}
   0 \to K(\cH) \to C^*(A(\cH)) \to C(\partial \bB_d) \to 0 .
\end{equation}
Here, $K(\mathcal{H})$ denotes the ideal of compact operators on $\mathcal{H}$, the first map is the inclusion
map, and the second map sends $M_f + K$ to $f \big|_{\partial \mathbb{B}_d}$ for $f \in A(\mathcal{H})$
and $K \in K(\mathcal{H})$.

We may identify an element of the multiplier algebra $\Mult(\cH)$ with its multiplication operator
on $\cH$ and thus regard $\Mult(\cH)$ as a subalgebra of $B(\cH)$. Then $\Mult(\cH)$ is WOT closed. 
Therefore by trace duality, 
\[ \Mult(\cH) = (\cT(\cH) / \Mult(\cH)_\bot)^* .\]
We write $\Mult(\cH)_* = \cT(\cH) / \Mult(\cH)_\bot$ and call this space the standard predual
of $\Mult(\cH)$. On bounded subsets of $\Mult(\cH)$, the corresponding weak-$*$ topology
coincides with the topology of pointwise convergence on $\bB_d$.
One direction follows because 
\[ f(w) = \langle M_f k_w, k_w \|k_w\|^{-2} \rangle \qquad\text{for } w \in \bB_d ,\]
where $k_w = K(\cdot,w)$ denotes the reproducing kernel at $w$.
The converse follows because the kernel functions span $\cH$.

A reproducing kernel Hilbert space $\cH$ on $\bB_d$ with kernel $K$ is said to be a Pick space if the analogue
of Pick's interpolation theorem holds; see \cite{AM02} for background.
More precisely,
we say that $\cH$ satisfies the $k$-point Pick property if whenever $z_1,\ldots,z_k \in \bB_d$ and $w_1,\ldots,w_k \in \bC$ so that
\begin{equation*}
  \big[ K(z_i,z_j) (1 - w_i \overline{w_j}) \big]_{i,j=1}^k \ge 0,
\end{equation*}
then there exists a multiplier $f \in \Mult(\cH)$ of norm at most one so that $f(z_i) = \lambda_i$
for $1 \le i \le k$. If $\cH$ satisfies the $k$-point Pick property for all $k \in \bN$,
we say that $\cH$ is a Pick space.

It is frequently useful to allow matrix valued targets. In this setting,
$\cH$ is said to satisfy the $M_n$-Pick property if whenever $k \in \bN$,
$z_1,\ldots,z_k \in \bB_d$ and $W_1,\ldots,W_k \in M_n(\bC)$ so that
\begin{equation*}
  \big[ K(z_i,z_j) (I_n - W_i W_j^*) \big]_{i,j=1}^k \ge 0,
\end{equation*}
then there exists $F \in M_n(\Mult(\cH))$ of multiplier norm at most $1$ so that $F(z_i) = W_i$
for $1 \le i \le k$. If $\cH$ satisfies the $M_n$-Pick property for all $n \in \bN$, then $\cH$
is said to be a complete Pick space. The prototypical example of a complete Pick space is $H^2(\bD)$.
Other examples include the Drury--Arveson space $H^2_d$, the classical Dirichlet space,
and more generally the spaces $\cH_s(\bB_d)$ for $s \le 0$.

\subsection{Operator space basics}
\label{ss:os}

By definition, $A(\cH)$ is a non-selfadjoint operator algebra,
hence it is natural to consider the sequence of matrix norms
on $A(\cH)$ and prove that identifications involving $A(\cH)$ are not just
isometric isomorphisms, but completely isometric isomorphisms.
For our concrete interpolation problem, this will translate into interpolating
matrix valued targets. Indeed, in the theory of Pick interpolation,
it is customary and sometimes necessary to consider matrix valued multipliers,
see \cite{AM02}.

We therefore recall the necessary basics from the theory of operator spaces. Our standard
reference will be \cite{BL04}. A concrete operator space is a subspace $X \subset B(\cH)$.
The identification $M_n(X) \subset M_n(B(\cH)) = B(\cH^n)$ endows each space $M_n(X)$
with a norm. An abstract operator space is a vector space $V$, together with
a sequence of norms on $M_n(V)$, satisfying certain axioms.
We will not require the axioms themselves and simply refer to \cite[1.2.12]{BL04}.

If $X$ and $Y$ are abstract operator spaces, then each linear map $\Phi: X \to Y$
induces linear maps $\Phi^{(n)}: M_n(X) \to M_n(Y)$ by applying $\Phi$ to each matrix entry.
One says that $\Phi$ is completely bounded if
\begin{equation*}
  \|\Phi\|_{cb} = \sup_{n \ge 1} \|\Phi^{(n)} \| < \infty,
\end{equation*}
and completely contractive if $\|\Phi\|_{cb} \le 1$.
We write $CB(X,Y)$ for the space of completely bounded maps from $X$ to $Y$,
endowed with $\|\cdot\|_{cb}$.
Moreover, $\Phi$ is completely isometric if each $\Phi^{(n)}$ is an isometry.
Similarly, $\Phi$ is a complete quotient map if each $\Phi^{(n)}$ is a quotient map, meaning
$\Phi^{(n)}$ maps the open unit ball of $M_n(X)$ onto the open unit ball of $M_n(Y)$.
More generally, $\Phi$ is completely surjective if there exists a constant $C > 0$ so that for all $n \in \bN$
and for all $y \in M_n(Y)$, there exists $x \in M_n(X)$ with $\Phi^{(n)}(x) = y$
and $\|x\| \le C \|y\|$.

If $X$ is an abstract operator space, then the dual space $X^*$ carries
a natural operator space structure, obtained by the identification
$M_n(X^*) = \CB(X,M_n)$, see \cite[1.2.20]{BL04}.
We will frequently use the fact that a linear map $\Phi: X \to Y$ is a complete isometry
if and only if the adjoint $\Phi^*: Y^* \to X^*$ is a complete quotient map.
Moreover, if $X$ and $Y$ are complete, then
$\Phi: X \to Y$ is a complete quotient map if and only if $\Phi^*: Y^* \to X^*$
is a complete isometry; see \cite[1.4.3]{BL04}.

If $X$ and $Y$ are operator spaces, $X \oplus_\infty Y$ denotes the direct
sum of $X$ and $Y$, equipped with the operator space structure defined
by
\begin{equation*}
  \| (x,y) \|_{M_n(X \oplus Y)} = \max ( \|x\|_{M_n(X)}, \|y\|_{M_n(Y)}),
\end{equation*}
see \cite[1.2.17]{BL04}. We also require the $1$-direct sum $X \oplus_1 Y$,
see \cite[1.4.13]{BL04} or \cite[Section 2.6]{Pisier03}.
The norm on $X \oplus_1 Y$ itself is simply the usual $1$-norm given by
$\|(x,y)\|_1 = \|x\| + \|y\|$.
The operator space structure on $X \oplus_1 Y$ is slightly more difficult to describe.
It is characterized by the following universal property: For any operator space $Z$
and any pair of complete contractions $\Phi: X \to Z$ and $\Psi: Y \to Z$, the map
\begin{equation*}
  X \oplus_1 Y \to Z, \quad (x,y) \mapsto \Phi(x) + \Psi(y),
\end{equation*}
is a complete contraction.
Moreover, the completely isometric identities 
\[
 (X \oplus_1 Y)^* = X^* \oplus_\infty Y^* \qand (X \oplus_\infty Y)^* = X^* \oplus_1 Y^* 
\]
hold.
It also follows from this duality that for any pair
of complete isometries $\Phi: X_1 \to X_2$ and $\Psi: Y_1 \to Y_2$,
the direct sum
\begin{equation*}
  \Phi \oplus \Psi: X_1 \oplus_1 Y_1 \to X_2 \oplus_1 Y_2
\end{equation*}
is also a complete isometry.

We will apply these considerations to $A(\cH)$ and $C(X)$, the space of continuous
functions on a compact metric space $X$. By definition, $A(\cH)$ carries an operator space
structure by identifying a function in $A(\cH)$ with the corresponding multiplication
operator on $\cH$. We will endow $A(\cH)^*$ with the dual operator space structure.
Similarly, $\Mult(\cH)^*$ carries the dual operator space structure,
which also gives $\Mult(\cH)_* \subset \Mult(\cH)^*$ the structure of an operator space.
We claim that with this definition, the operator space dual of $\Mult(\cH)_*$
is $\Mult(\cH)$. Indeed, since the inclusion of an operator
space into its second dual is a complete isometry (see \cite[Proposition 1.4.1]{BL04}),
it suffices to observe that there exists an operator space structure
on $\Mult(\cH)_*$ whose operator space dual is $\Mult(\cH)$.
But this follows from the concrete description $\Mult(\cH)_* = \cT(\cH) / \Mult(\cH)_\bot$,
which allows us to endow $\Mult(\cH)_*$ with the quotient
operator space structure of $\cT(\cH)$; see \cite[Lemma 1.4.6]{BL04}.

Moreover, if $X$ is a compact metric space, then $C(X)$ is endowed
with the operator space structure given by the identification $M_n(C(X)) = C(X,M_n)$.
Finally, by the Riesz representation theorem, $M(X) = C(X)^*$, which allows us to equip
$M(X)$ with the dual operator space structure.

\subsection{Henkin measures, totally singular measures and totally null sets}

Throughout, let $\cH$ be a regular unitarily invariant space on $\bB_d$.
Recall from Subsection \ref{ss:RKHS} that $\Mult(\cH)$ is the dual space of $\Mult(\cH)_* = \cT(\cH) / \Mult(\cH)_\bot$, and that on bounded subsets of $\Mult(\cH)$, the corresponding
weak-$*$ topology coincides with the topology of pointwise convergence on $\bB_d$.
A linear functional $\varphi \in A(\cH)^*$ is said to be a $\Mult(\cH)$-Henkin
functional if it extends to a weak-$*$ continuous functional on $\Mult(\cH)$.
The following characterization of Henkin functionals is a special case of \cite[Lemma 3.1]{BHM17}.

\begin{lem}
  \label{lem:Henkin_sequence}
  Let $\cH$ be a regular unitarily invariant space. Then the following assertions
  are equivalent for a functional $\varphi \in A(\cH)^*$.
  \begin{enumerate}[label=\normalfont{(\roman*)}]
    \item $\varphi$ is $\Mult(\cH)$-Henkin.
    \item Whenever $(p_n)$ is a sequence of polynomials so that \mbox{$\|p_n\|_{\Mult(\cH)} \le 1$}
      for all $n \in \bN$ and $\lim_{n \to \infty} p_n(z) = 0$ for all $z \in \bB_d$, then also
      $\lim_{n \to \infty} \varphi(p_n) = 0$.
  \end{enumerate}
\end{lem}

In particular, examples of Henkin functionals are evaluations at points in $\bB_d$.
We will use the following lemma in a few places. It is entirely routine,
but since we do not have an explicit reference for the exact statement, we provide the proof.

Let $\cU_d$ denote the group of $d \times d$-complex unitary matrices acting on the ball $\bB_d$ in the usual manner.
If $F: \bB_d \to M_n$, $0 \le r \le 1$ and $U \in \cU_d$, define $F_{r,U}(z) = F(r Uz)$.

\begin{lem}
  \label{lem:multiplier_rotations}
  Let $\cH$ be a regular unitarily invariant space.
  \begin{enumerate}[label=\normalfont{(\alph*)}]
    \item
      For each $U \in \cU_d$, the map
      \begin{equation*}
        \Mult(\cH) \to \Mult(\cH), \quad f \mapsto f_{r,U},
      \end{equation*}
      is a complete isometry if $r=1$ and a complete contraction when $0 \le r<1$.
      Moreover, it is weak-$*$--weak-$*$ continuous and maps $A(\cH)$ into $A(\cH)$.
    \item For each $f \in \Mult(\cH)$, the map
      \begin{equation*}
        [0,1] \times \cU_d  \to \Mult(\cH), \quad (r,U) \mapsto f_{r,U},
      \end{equation*}
      is weak-$*$ continuous.
  \end{enumerate}
\end{lem}

\begin{proof}
  (a) Since $\cH$ is unitarily invariant, there exists an SOT-continuous unitary representation
  \begin{equation*}
    \Gamma: \cU_d \to B(\cH), \quad \Gamma(U)(g) = g_{1,U}.
  \end{equation*}
  If $f \in \Mult(\cH)$, then $M_{f_{1,U}} = \Gamma(U) M_f \Gamma(U^*)$,
  which proves that the map $f \mapsto f_{1,U}$ is a complete isometry.
  
  To prove that the map is a complete contraction if $r < 1$, it suffices to show that for each $r \in [0,1)$, the map $f \mapsto f_{r,I_d}$
  is completely contractive on $\Mult(\cH)$.
  From the above argument, we see that the map
  \begin{equation*}
    \cU_d \to \Mult(\cH), \quad U \mapsto f_{1,U},
  \end{equation*}
  is SOT-continuous for each $f \in \Mult(\cH)$.
  If $P_r(\lambda) = \frac{1 - r^2}{|1 - r \lambda|^2}$ denotes the Poisson kernel,
  and if $F \in M_n(\Mult(\cH))$,
  then the integral
  \begin{equation*}
    \int_{\bT} F_{1, \lambda I_d} P_r(\lambda) \, dm (\lambda),
  \end{equation*}
  where $m$ denotes the normalized Lebesgue measure on $\mathbb{T}$,
  converges in the strong operator topology.
  Evaluating at $z \in \bB_d$, we find that the integral equals $F_{r,I_d}$.
  Thus, standard properties of the Poisson kernel yield that
  \begin{equation*}
    \|F_{r,I_d}\|_{M_n(\Mult(\cH))} \le \|F\|_{M_n(\Mult(\cH))}.
  \end{equation*}

  To establish weak-$*$--weak-$*$ continuity, an application of the Krein--Smulian
  theorem shows that it suffices to establish weak-$*$--weak-$*$ continuity
  on bounded subsets of $\Mult(\cH)$.
  But since the map $f \mapsto f_{r,U}$ is bounded, this follows
  from the fact that on bounded subsets of $\Mult(\cH)$, the weak-$*$ topology
  agrees with the topology of pointwise convergence on $\bB_d$.

  Finally, the map takes polynomials to polynomials, so by continuity, it maps $A(\cH)$
  into $A(\cH)$.

  (b) From (a), it follows that the set
  \begin{equation*}
    \{ f_{r,U} : 0 \le r \le 1,\ U \in \cU_d \}
  \end{equation*}
  is a bounded subset of $\Mult(\cH)$. Therefore the statement follows once again from the fact
  that on bounded subsets of $\Mult(\cH)$, the weak-$*$ topology
  agrees with the topology of pointwise convergence on $\bB_d$.
\end{proof}

The following lemma shows that the space of Henkin functionals on $A(\cH)$ can be identified with $\Mult(\cH)_*$.
Without the statement about complete isometry, it
is contained in Lemma 3.1 of \cite{BHM17} and its proof.
The proof carries over with minimal changes.

\begin{lem}
  \label{lem:predual_A(H)}
  Let $\cH$ be a regular unitarily invariant space. The map
  \begin{equation*}
    \Mult(\cH)_* \to A(\cH)^*, \quad \varphi \mapsto \varphi \big|_{A(\cH)},
  \end{equation*}
  is a complete isometry whose range is the space of $\Mult(\cH)$-Henkin functionals.
  In particular, the space of $\Mult(\cH)$-Henkin functionals is norm closed in $A(\cH)^*$.
\end{lem}

\begin{proof}
It is clear from the definition of Henkin functionals that the restriction map takes $\Mult(\cH)_*$ onto the set of Henkin functionals.
  Since the inclusion $A(\cH) \subset \Mult(\cH)$ is a complete isometry,
  the restriction map
  \begin{equation*}
    \Mult(\cH)_* \subset \Mult(\cH)^* \to A(\cH)^*, \quad \varphi \mapsto \varphi \big|_{A(\cH)},
  \end{equation*}
  is a complete contraction. To see that it is completely isometric, it suffices to observe
  that for each $n \in \bN$, the unit ball of $M_n(A(\cH))$ is weak-$*$ dense
  in the unit ball of $M_n(\Mult(\cH))$.
  To see this, note that if $F \in M_n(\Mult(\cH))$,  then
  $\|F_{1, \lambda I_d}\|_{M_n(\Mult(\cH))} = \|F\|_{M_n(\Mult(\cH))}$ for all $\lambda \in \bT$,
  and the map $\lambda \mapsto F_{1, \lambda I_d}$ is continuous in the weak-$*$ topology
  by Lemma \ref{lem:multiplier_rotations}.
  In this setting, standard properties of the Fej\'er kernel (cf. \cite[Lemma I 2.5]{Katznelson04})
  imply that the Fej\'er means $(F_n)$ of $F$
  are matrices of polynomials, satisfy $\|F_n\|_{M_n(\Mult(\cH))} \le \|F\|_{M_n(\Mult(\cH))}$
  for all $n \in \bN$ and converge to $F$ in the weak-$*$ topology.
\end{proof}

In the sequel, we will therefore identify $\Mult(\cH)_*$ with a subset of $A(\cH)^*$.

As explained in the introduction,
a regular Borel measure $\mu$ on $\partial \bB_d$ is said to be $\Mult(\cH)$-Henkin
if the associated integration functional
\begin{equation*}
  \rho_\mu: A(\cH) \to \bC, \quad f \mapsto \int_{\partial \bB_d} f \, d \mu,
\end{equation*}
is $\Mult(\cH)$-Henkin. We write $\HENH$ for the space of all $\Mult(\cH)$-Henkin
measures on $\partial \bB_d$.
We say that $\mu$ is $\Mult(\cH)$-totally singular
if it is singular with respect to every $\Mult(\cH)$-Henkin measure.
The space of all $\Mult(\cH)$-totally singular measures is denoted by $\TSH$.

If $X$ is a compact metric space, a norm closed subspace $\Sigma \subset M(X)$
is called a \emph{band} if whenever $\mu \in \Sigma$ and $\nu \in M(X)$
is absolutely continuous with respect to $\mu$, then $\nu \in \Sigma$. In particular,
$\mu \in \Sigma$ if and only if $|\mu| \in \Sigma$.
The sets $\HENH$ and $\TSH$ form complementary bands in $M(\partial \bB_d)$.
In particular:

\begin{lem}
  \label{lem:hen_band}
  $\HENH$ and $\TSH$ are bands in $M(\partial \bB_d)$. Every $\mu \in M(\partial \bB_d)$
  has a unique decomposition $\mu = \mu_a + \mu_s$,
  where $\mu_a \in \HENH$ and $\mu_s \in \TSH$.
\end{lem}

\begin{proof}
  Since $A(\cH)$ is contractively contained in $C( \partial \bB_d)$, the map
  \begin{equation*}
    M( \partial \bB_d) \to A(\cH)^*, \quad \mu \mapsto \rho_{\mu},
  \end{equation*}
  is contractive. Lemma \ref{lem:predual_A(H)} shows that the space of Henkin functionals is closed in $A(\cH)^*$,
  so it follows that $\HENH$ is closed in $M( \partial \bB_d)$.
  The band property of $\HENH$ was shown in \cite[Lemma 3.3]{BHM17} (see
  also \cite[Theorem 5.4]{CD16} for the case of the Drury--Arveson space).
  It is a general fact that for every non-empty subset $A$ of measures
  on a compact metric space, the set
  \begin{equation*}
    A^\bot = \{ \nu \in M(X): \nu \bot \mu \tfa \mu \in A \}
  \end{equation*}
  is a norm closed band; see \cite[Remark II 2.3]{BK77}. In particular,
  $\TSH = \HENH^\bot$ is a band. The statement about the decomposition
  then follows from the F.\ Riesz decomposition theorem for bands;
  see \cite{KS69} or \cite[Section II.2]{BK77}.
\end{proof}

In the theory of classical Henkin measures on the ball (see \cite[Chapter 9]{Rudin08}),
one obtains an a priori stronger notion of singularity in the decomposition
by applying the Glicksberg--K\"onig--Seever decomposition theorem.
This theorem is applied to the weak-$*$ compact convex set of representing measures of the origin.
The key points are theorems of Henkin and Cole--Range,
which show that a measure is a classical Henkin measure
if and only it is absolutely continuous with respect to some representing measure
of the origin.

For more general reproducing kernel Hilbert spaces $\cH$, we do not know if there exists a similar characterization
of $\Mult(\cH)$-Henkin measures. Nevertheless, we will show in Proposition
\ref{prop:measure_concentrated} as a consequence of Theorem \ref{thm:A_H_decomp_intro}
that $\Mult(\cH)$-totally singular measures satisfy a similar strong singularity property.

A Borel subset $E \subset \partial \bB_d$ is said to be $\Mult(\cH)$-totally null
if $|\mu|(E) = 0$ for all $\Mult(\cH)$-Henkin measures $\mu$.
The following equivalent characterizations follow from
the fact that $\Mult(\cH)$-Henkin measures form a band.

\begin{lem}
  \label{lem:TN_elementary}
  Let $E \subset \partial \bB_d$ be a Borel set. The following assertions are equivalent:
  \begin{enumerate}[label=\normalfont{(\roman*)}]
    \item $E$ is $\Mult(\cH)$-totally null, that is, $|\mu|(E) = 0$ for all $\Mult(\cH)$-Henkin
      measures $\mu$.
    \item If $\nu \in M(\partial \bB_d)$ is concentrated on $E$, then $\nu \in \TSH$.
    \item If $\mu$ is a positive $\Mult(\cH)$-Henkin measure
      that is concentrated on $E$, then $\mu =0$.
  \end{enumerate}
\end{lem}

\begin{proof}
  (i) $\Rightarrow$ (ii) If $\nu$ is concentrated on $E$ and $\mu$ is $\Mult(\cH)$-Henkin,
  then $\nu \bot \mu$ by (i), so $\nu \in \TSH$.
  
  (ii) $\Rightarrow$ (iii) Let $\mu$ be a $\Mult(\cH)$-Henkin measure that is concentrated on $E$.
  Then $\mu \in \TSH$ by (ii), so $\mu \bot \mu = 0$ and hence $\mu  = 0$.

  (iii) $\Rightarrow$ (i) Let $\mu$ be $\Mult(\cH)$-Henkin. We wish
  to show that $|\mu|(E) = 0$. Since $\Mult(\cH)$-Henkin
  measures form a band by Lemma \ref{lem:hen_band}, the measure $|\mu|$ is $\Mult(\cH)$-Henkin
  as well, so we may assume that $\mu$ is a positive measure. Define $\mu|_E \in M(\partial \bB_d)$
  by $\mu|_E(A) = \mu(A \cap E)$ for Borel sets $A \subset \partial \mathbb{B}_d$. Then $\mu|_E$ is absolutely continuous with respect to $\mu$,
  so using again that $\Mult(\cH)$-Henkin measures form a band, we see that $\mu|_E$
  is $\Mult(\cH)$-Henkin. By (iii), we obtain that  $\mu(E) = \mu|_E(E) = 0$.
\end{proof}

In concrete examples of spaces $\cH$, there may be other notions of smallness on $\partial \bB_d$.
For instance, in the Dirichlet space $\cD$, a classical notion
of smallness is that of logarithmic capacity zero, see \cite[Chapter 2]{EKM+14}.
A positive measure $\mu \in M(\bT)$ is said to have finite energy
if the functional
\begin{equation*}
  \bC[z] \to \bC, \quad p \mapsto \int_{\bT} p \, d \mu,
\end{equation*}
is continuous with respect to the norm of $\cD$. (This is not the usual potential
theoretic definition, but it is an equivalent one, see \cite[Theorem 2.4.4]{EKM+14}.)
A compact subset $E \subset \bT$ has logarithmic capacity zero if and only if it does not
support a non-zero positive measure of finite energy.
Thus, we obtain the following implication between the two notions of smallness.

\begin{prop}
  \label{prop:TN_cap0}
  Let $E \subset \bT$ be compact. If $E$ is $\Mult(\cD)$-totally null, then $E$ has logarithmic
  capacity zero.
\end{prop}

\begin{proof}
  If $E$ has non-zero logarithmic capacity, then it supports a non-zero positive measure $\mu$ of finite energy.
  Clearly, $\mu$ is in particular $\Mult(\cD)$-Henkin, and $\mu(E) = \mu(\bT) \neq 0$, so $E$ is not $\Mult(\cD)$-totally null.
\end{proof}

While the present paper was being refereed, it was shown by Chalmoukis and the second author
that in certain Dirichlet type spaces on the ball, the totally null condition
is equivalent to a suitable capacity zero condition \cite{CH20}.
In particular, this provides the converse of Proposition \ref{prop:TN_cap0}.

\section{The dual space of \texorpdfstring{$A(\cH)$}{A(H)}}

We require a version of the F.\ Riesz decomposition theorem for representations
of $C(X)$.
Let $X$ be a compact metric space, let $\Sigma \subset M(X)$ be a norm closed
band of measures and let $\Sigma^\bot$ be the complementary band. Thus,
$\mu \in \Sigma^\bot$ if and only if $\mu$ is singular with respect to every measure in $\Sigma$.
We say that a unital $*$-representation $\pi: C(X) \to B(\cH)$ is $\Sigma$-absolutely continuous
if for all $x,y \in \cH$, the representing measure of the functional
$f \mapsto  \langle \pi(f)x,y \rangle$ belongs to $\Sigma$.

\begin{lem}
  \label{lem:band_decom}
  Let $X$ be a compact metric space, let $\Sigma \subset M(X)$ be a norm closed
  band of measures and let $\pi: C(X) \to B(\cH)$ be a unital $*$-representation.
  Then $\pi$ is unitarily equivalent to a direct sum of representations $\pi_a \oplus \pi_s$,
  where $\pi_a$ is $\Sigma$-absolutely continuous and $\pi_s$ is $\Sigma^\bot$-absolutely
  continuous.
\end{lem}

\begin{proof}
  There exists a set $S$ of regular Borel probability measures on $X$
  so that $\pi$ is unitarily equivalent to $\bigoplus_{\mu \in S} \pi_\mu$,
  where
  \begin{equation*}
    \pi_\mu: C(X) \to B(L^2(\mu)), \quad f \mapsto M_f^{\mu},
  \end{equation*}
  and $M_f^{\mu}$ denotes the operator of multiplication by $f$ on $L^2(\mu)$;
  see \cite[Section II.1]{Davidson96}.
  For each $\mu \in S$, we apply the decomposition theorem
  of F. Riesz (see \cite{KS69} or \cite[Section II.2]{BK77}) to write
  $\mu = \mu_a + \mu_s$, where $\mu_a \in \Sigma$ and $\mu_s \in \Sigma^\bot$.
  Since $\mu_a \bot \mu_s$, the representation $\pi_\mu$ decomposes
  as a direct sum $\pi_{\mu_a} \oplus \pi_{\mu_s}$. Let
  $\pi_a = \bigoplus_{\mu \in S} \pi_{\mu_a}$ and $\pi_s = \bigoplus_{\mu \in S} \pi_{\mu_s}$.
  Then $\pi$ is unitarily equivalent to $\pi_a \oplus \pi_s$.

  To see that $\pi_a$ is $\Sigma$-absolutely continuous, note that if $g = (g_\mu)$
  and $h = (h_\mu)$ belong to $\bigoplus_{\mu \in S} L^2( \mu_a)$, then
  \begin{equation*}
    \langle \pi_a(f) g,h \rangle = \sum_{\mu \in S} \int_X f g_\mu \ol{h_\mu} \, d \mu_a
    \quad \text{ for all } f \in C(X).
  \end{equation*}
  For each $\mu \in S$, the measure $g_\mu \overline{h_\mu} d \mu_a$ belongs
  to $\Sigma$ since $\Sigma$ is a band. Moreover, the Cauchy--Schwarz inequality implies that the
  series $\sum_{\mu \in S} g_\mu \overline{h_\mu} d \mu_a$ converges absolutely
  to a measure $\nu$ in the Banach space $M(\partial \bB_d)$. Since $\Sigma$
  is norm closed, $\nu \in \Sigma$, and
  \begin{equation*}
    \langle \pi_a(f) g , h \rangle = \int_X f \, d \nu ; 
  \end{equation*}
  whence $\pi_a$ is $\Sigma$-absolutely continuous.
  The same argument shows that $\pi_s$ is $\Sigma^\bot$-absolutely continuous.
\end{proof}

We now prove a more precise version of Theorem \ref{thm:A_H_decomp_intro}.
Recall from Subsection \ref{ss:os} that $\Mult(\cH)_*,A(\cH)^*$ and $\TSH \subset M(\partial \bB_d)$
are equipped with their respective dual operator space structures.

\begin{thm}
  \label{thm:A_H_decomp}
  Let $\cH$ be a regular unitarily invariant space on $\bB_d$. Then
  the natural map
  \begin{equation*}
    \Mult(\cH)_* \oplus_1 \TSH \to A(\cH)^*, \quad (\varphi, \nu) \mapsto \varphi \big|_{A(\cH)}  + \rho_\nu ,
  \end{equation*}
  is a completely isometric isomorphism.
\end{thm}

\begin{proof}
  Let $\Phi$ denote the linear map in the statement.
  Since the inclusion $A(\cH) \subset C(\partial \bB_d)$ is completely contractive,
  the map $\TSH \to A(\cH)^*, \nu \mapsto \rho_\nu$ is completely contractive.
  We already saw in Lemma \ref{lem:predual_A(H)} that the map
  \begin{equation*}
    \Mult(\cH)_* \subset \Mult(\cH)^* \to A(\cH)^*, \quad \varphi \mapsto \varphi \big|_{A(\cH)},
  \end{equation*}
  is a complete isometry. Thus,
  $\Phi$ is completely contractive by the universal property of the $1$-direct sum.
  To see that
  $\Phi$ is injective, let $\varphi \in \Mult(\cH)_*$ and $\nu \in \TSH$ so that
  $\varphi + \rho_\nu = 0 \in A(\cH)^*$. Then $\rho_\nu = - \varphi \in \Mult(\cH)_*$,
  so the measure $\nu$ is $\Mult(\cH)$-Henkin. By definition of $\TSH$, it follows
  that $\nu \bot \nu$, so that $\nu = 0$ and hence also $\varphi = 0$.

  It remains to show that if $\varphi \in M_n(A(\cH)^*)$ has norm $1$, then it has a preimage under $\Phi$
  of norm at most $1$. We first sketch the idea of the proof, assuming for simplicity that $n=1$.
  If $\varphi \in A(\mathcal{H})^*$, then by the Hahn--Banach theorem, $\varphi$ extends to a functional
  on the Toeplitz $C^*$-algebra $C^*(A(\mathcal{H})) \subset B(\mathcal{H})$.
  Since $\cH$ is regular, \cite[Theorem 4.6]{GHX04} shows that there exists a short exact
  sequence
  \begin{equation}
    \label{eqn:short_exact_proof}
    0 \to K(\cH) \to C^*(A(\cH)) \to C(\partial \bB_d) \to 0,
  \end{equation}
  where the first map is the inclusion and the second map sends $M_f + K$ to $f \big|_{\partial \bB_d}$ for $f \in A(\cH)$ and $K \in K(\cH)$.
  This short exact sequence will imply that $\varphi$ decomposes as $\varphi = \varphi_1 + \varphi_2$,
  where $\varphi_1$ extends to a weak-$*$ continuous functional on $B(\mathcal{H})$,
  and $\varphi_2$ is induced by a functional on $C(\partial \mathbb{B}_d)$,
  so $\varphi_2 = \rho_\mu$ is an integration functional for some measure $\mu \in M(\partial \mathbb{B}_d)$.
  Further decomposing the measure $\mu$ into its $\Mult(\mathcal{H})$-Henkin part
  $\mu_a$ and its $\Mult(\mathcal{H})$-totally singular part $\mu_s$,
  we find that $\varphi = (\varphi_1 + \rho_{\mu_a}) + \rho_{\mu_s}$
  is a sum of a functional in $\Mult(\mathcal{H})_*$ and a functional given by a $\Mult(\mathcal{H})$-totally
  singular measure.
  To adapt this idea to elements of $M_n(A(\mathcal{H}))^*$, we will use dilation theory.
  
  Let now $\varphi \in M_n(A(\cH)^*) = \CB(A(\cH),M_n) $ with $\|\varphi\|_{cb} = 1$.
  By the Haagerup--Paulsen--Wittstock
  dilation theorem (Theorems 8.2 and 8.4 in \cite{Paulsen02}), there exist a unital $*$-representation $\pi: C^*(A(\cH)) \to B(\cK)$ on a
  separable Hilbert space $\cK$ and contractions $V,W: \bC^n \to \cK$
  so that $\varphi(f) = W^* \pi(f) V$ for all $f \in A(\cH)$.
  From the short exact sequence \eqref{eqn:short_exact_proof} and
 from a well known result about representations of $C^*$-algebras,
 it follows that $\pi$ splits as an orthogonal direct sum $\pi = \pi_1 \oplus \pi_2$, where $\pi_1$
is unitarily equivalent to a multiple of the identity representation and $\pi_2|_{K(\cH)} = 0$. 
Thus $\pi_2$ factors through the quotient map onto $C(\partial \bB_d)$, and so it can be regarded as a representation of $C(\partial \bB_d)$;
see, for instance, the discussion preceding \cite[Theorem 1.3.4]{Arveson76}.
Lemma \ref{lem:hen_band} and Lemma \ref{lem:band_decom} show that
the representation $\pi_2$ of $C(\partial \bB_d)$ further splits as $\pi_a \oplus \pi_s$
acting on $\cL_a \oplus \cL_s$,
where $\pi_a$ is $\HENH$-absolutely continuous
and $\pi_s$ is $\TSH$-absolutely continuous.

Thus, there exist a countable cardinal $\kappa$ and contractions
\begin{equation*}
  \begin{bmatrix}
    V_1 \\ V_2 \\ V_3
  \end{bmatrix},
  \begin{bmatrix}
    W_1 \\ W_2 \\ W_3
  \end{bmatrix}: \bC^n \to \cH^\kappa \oplus \cL_a \oplus \cL_s
\end{equation*}
so that
\begin{equation*}
  \varphi(f) = W_1^* M_f^{\kappa} V_1 + W_2^* \pi_a(f) V_2 + W_3^* \pi_s(f) V_3 \quad
  (f \in A(\cH)).
\end{equation*}
Define
\begin{equation*}
  \varphi_a: A(\cH) \to M_n, \quad 
  f \mapsto W_1^* M_f^{\kappa} V_1 + W_2^* \pi_a(f) V_2 ,
\end{equation*}
and
\begin{equation*}
  \varphi_s: C(\partial \bB_d) \to M_n, \quad f \mapsto W_3^* \pi_s(f) V_3.
\end{equation*}
Then for all $x,y \in \bC^n$, the functional $f \mapsto \langle \varphi_a(f) x,y \rangle$
belongs to $\Mult(\cH)_*$ and the representing measure of $f \mapsto \langle \varphi_s(f) x,y \rangle$
belongs to $\TSH$. Hence $\varphi_a \in M_n(\Mult(\cH)_*)$ and $\varphi_s$ is
an $n \times n$ matrix of integration functionals given by elements of $\TSH$.

We finish the proof by showing that $(\varphi_a ,\varphi_s)$ has
norm at most $1$ in $M_n(\Mult(\cH)_* \oplus_1 \TSH)$. To see this,
observe that by Lemma \ref{lem:hen_band}, we have the isometric inclusion
\begin{align*}
  M_n( \Mult(\cH)_* \oplus_1 \TSH) &\subset M_n( A(\cH)^* \oplus_1 C(\partial \bB_d)^*) \\
  &= \CB( A(\cH) \oplus_\infty C(\partial \bB_d),M_n),
\end{align*}
so we have to show that the map
\begin{equation*}
  A(\cH) \oplus_{\infty} C(\partial \bB_d) \to M_n, \quad (f,g)
  \mapsto \varphi_a(f) + \varphi_s(g),
\end{equation*}
is completely contractive. But this is immediate from the description
\begin{equation*}
  \varphi_a(f) + \varphi_s(g) =
  \begin{bmatrix}
    W_1^* & W_2^* & W_3^*
  \end{bmatrix}
  \begin{bmatrix}
    M_f^{\kappa} & 0 & 0 \\
    0 & \pi_a(f) &  0 \\
    0 & 0 & \pi_s(g)
  \end{bmatrix}
  \begin{bmatrix}
    V_1 \\ V_2 \\ V_3
  \end{bmatrix}
\end{equation*}
for all $f\in A(\cH)$ and $g \in C(\partial \bB_d)$.
\end{proof}

As explained in the introduction, the preceding result applies in particular to the Hardy space $H^2(\bB_d)$,
yielding the decomposition of the dual of the ball algebra explained in \cite[Chapter 9]{Rudin08},
and to the Drury--Arveson space, thus providing another proof of Corollary 4.3 and Theorem 4.4 of \cite{CD16}.
But it also applies to the Dirichlet space, and more generally the spaces $\cH_s(\bB_d)$.

\begin{rem}
  \label{rem:everything_henkin}
If the coefficients $(a_n)$ in the reproducing kernel $K(z,w) = \sum_{n=0}^\infty a_n \langle z,w \rangle^n$
are summable, then $\cH$ is continuously contained in $C(\ol{\bB_d})$. Hence every measure in $M(\partial \bB_d)$ is $\Mult(\cH)$-Henkin
and $\TSH = \{0\}$. Thus, by Theorem \ref{thm:A_H_decomp}, $A(\cH)^* = \Mult(\cH)_*$, and every functional
on $A(\cH)$ extends to a weak-$*$ continuous functional on $\Mult(\cH)$ in this setting.

Such examples occur among the spaces $\cH_s(\mathbb{D})$, namely if $s < -1$.
In fact, in this case, we have that $\cH_s(\bD)
= \Mult(\cH_s(\bD))$ with equivalence of norms for $s < -1$ (see Proposition 31 and Example 1 on page 99 of \cite{Shields74}).
Hence also $A(\cH_s(\mathbb{D})) = \Mult(\cH_s(\mathbb{D}))$ and $\Mult(\cH_s(\mathbb{D}))$
is a reflexive Banach space. So in this case, weak-$*$ continuous functionals on $\Mult(\cH)$
and functionals on $A(\cH)$ agree, and in fact every functional on $\Mult(\mathcal{H})$ is weak-$*$ continuous.

There are cases where $\Mult(\cH)_* = A(\cH)^*$, but not every functional on $\Mult(\cH)$ is weak-$*$ continuous.
The Salas space $\cH$, which was constructed in \cite{AHM+17a}, is a regular unitarily invariant complete Pick space on $\bD$
for which the corona theorem fails. In this space, $\sum_{n=0}^\infty a_n < \infty$, so $\cH \subset C(\ol{\bD})$ and $A(\cH)^* = \Mult(\cH)_*$,
but there are functionals on $\Mult(\cH)$ that are not weak-$*$ continuous.
For instance, there is a character $\chi$ on $\Mult(\cH)$ such that $\chi(p) = p(1)$ for all polynomials $p$,
but $\chi$ is not given by evaluation at $1$. Indeed, the equality $A(\cH)^* = \Mult(\cH)_*$ merely
implies that for every functional $\varphi$ on $\Mult(\cH)$, there exists a weak-$*$ continuous functional
on $\Mult(\cH)$ that agrees with $\varphi$ on $A(\cH)$. In the case of the character $\chi$, that functional is the functional
of point evaluation at $1$.
\end{rem}

\section{Totally singular measures}

The decomposition of Theorem \ref{thm:A_H_decomp} implies the following functional analytic characterization of $\TSH$. 
This is a special case of a general principle that is contained in a forthcoming paper of the second author with R.\ Clou\^atre.

\begin{lem}
  \label{lem:singular_fa}
  Let $\varphi \in A(\cH)^*$. Then $\varphi = \rho_\nu$ for some $\nu \in \TSH$
  if and only if there exists a sequence $(f_n)$ in the unit ball of $A(\cH)$ that
  tends to zero in the weak-$*$ topology of $\Mult(\cH)$ such that $\dlim_{n \to \infty} \varphi(f_n)
  = \|\varphi\|$. \vspace{-.3ex}
\end{lem}

\begin{proof}
  Suppose first that there exists a sequence $(f_n)$ in the unit ball of $A(\cH)$ that tends
  to zero in the weak-$*$ topology of $\Mult(\cH)$ with $\dlim_{n \to \infty} \varphi(f_n) = \|\varphi\|$.
  By Theorem \ref{thm:A_H_decomp}, there exist $\varphi_a \in \Mult(\cH)_*$ and $\nu \in \TSH$
  such that $\varphi = \varphi_a + \rho_\nu$ and $\|\varphi\| = \|\varphi_a\| + \|\rho_\nu\|$.
  Then $\lim_{n \to \infty} \varphi_a(f_n) = 0$, so
  \begin{equation*}
    \|\varphi_a\| + \|\rho_\nu\|
    = \|\varphi\| = \lim_{n \to \infty} \varphi(f_n)
    \le \limsup_{n \to \infty} |\rho_\nu(f_n)| \le \|\rho_\nu\|,
  \end{equation*}
  so $\varphi_a = 0$ and hence $\varphi = \rho_\nu$.

  Conversely, let $\varphi = \rho_\nu$ for some $\nu \in \TSH$. By Theorem \ref{thm:A_H_decomp},
  \begin{equation*}
    A(\cH)^{**} \cong \Mult(\cH) \oplus_{\infty} \TSH^*,
  \end{equation*}
  and with this identification, $\TSH^* = (\Mult(\cH)_*)^\bot$. Thus,
  by the Hahn--Banach theorem, there exists $\Lambda \in (\Mult(\cH)_*)^\bot \subset A(\cH)^{**}$
  of norm $1$ so that $\Lambda(\varphi) = \|\varphi\|$. Applying Goldstine's theorem (see \cite[Proposition V.4.1]{Conway90} or \cite[II.A.13]{Wojtaszczyk91}), we find a
  net $(g_\alpha)$ in the unit ball of $A(\cH)$ that converges to $\Lambda$ in the weak-$*$ topology
  of $A(\cH)^{**}$. In particular, $\|\varphi\| = \Lambda(\varphi) = \lim_\alpha \varphi(g_\alpha)$.
  Moreover, since $\Lambda \in (\Mult(\cH)_*)^\bot$, it follows that
  if $\psi \in \Mult(\cH)_*$, then $\lim_\alpha \psi(g_\alpha) = \Lambda(\psi) = 0$, so
  the net $(g_\alpha)$ converges to zero in the weak-$*$ topology of $\Mult(\cH)$.

  We finish the proof by using separability of $\cH$ to replace the net $(g_\alpha)$ with a sequence.
  To this end, note that since $\cH$ is separable, so is $\Mult(\cH)_*$, hence there
  exists a metric $d$ on the unit ball of $\Mult(\cH)$ that induces the weak-$*$ topology there.
  For each $n \in \bN$, the preceding paragraph shows that there
  exists an index $\alpha(n)$ so that $d(g_{\alpha(n)},0) < 1/n$ and
  $| \varphi(g_{\alpha(n)}) - \|\varphi\| | < 1 / n$. We set $f_n = g_{\alpha(n)}$. Then the sequence
  $(f_n)$ tends to zero in the weak-$*$ topology of $\Mult(\cH)$ and satisfies $\lim_{n \to \infty} \varphi(f_n) =
  \|\varphi\|$.
\end{proof}

\begin{rem}
We wish to point out a subtlety in the previous argument.
Note that Goldstine's theorem is applied to $A(\cH)^{**}$, but $A(\cH)^{*}$ is typically not separable, 
so the weak-* topology on the unit ball of $A(\cH)^{**}$ is not metrizable. 
Indeed, if there are $\Mult(\cH)$-totally null sets, then each point mass $\delta_z$ for 
$z \in \partial \bB_d$ provides an uncountable family of functionals which are distance $2$ apart.
The argument of the last paragraph only uses the separability of $\Mult(\cH)_*$.
\end{rem}
 
We record the following consequence.
\begin{cor}
 \label{C:singular_fa}
 Let $\nu \in \TSH$. Then there is a sequence $(f_n)$ in the unit ball of $A(\cH)$ that
 tends to zero in the weak-$*$ topology of $\Mult(\cH)$ and converges to $1$ $\nu$-almost everywhere.
\end{cor}

\begin{proof}
  The definition of $\TSH$ implies that if $\nu \in \TSH$, then so is $|\nu|$,
  hence we may assume that $\nu$ is a positive probability measure.
  Theorem \ref{thm:A_H_decomp} implies that $\|\rho_\nu\| = 1$, so by Lemma \ref{lem:singular_fa},
  there exists a sequence $(f_n)$ in the unit ball of $A(\cH)$ that tends to zero in the weak-$*$
  topology of $\Mult(\cH)$ such that
  \begin{equation}
    \label{eqn:convergence}
    1 = \lim_{n \to \infty} \int_{\partial \bB_d} f_n \, d \nu.
  \end{equation}
  Since the multiplier norm dominates the supremum norm, $|f_n| \le 1$ on $\partial \bB_d$, so
  \eqref{eqn:convergence} implies that $(f_n)$ converges to $1$ in $L^2(\nu)$.
  Indeed,
  \begin{equation*}
    \|f_n - 1\|_{L^2(\nu)}^2 = 1 + \|f_n\|_{L^2(\nu)}^2 - 2 \Re \int_{\partial \bB_d} f_n \, d\nu
    \le 2 - 2 \Re \int_{\partial \bB_d} f_n \, d \nu.
  \end{equation*}
  Hence, by passing to a subsequence, we may achieve that $(f_n)$ converges to $1$ pointwise
  $\nu$-almost everywhere on $\partial \bB_d$.
\end{proof}

Classically, an $H^\infty(\bB_d)$-totally singular measure is not just singular
with respect to each Henkin measure, but it is in fact concentrated on an $F_\sigma$ set
that is totally null, i.e.\ a set that is null simultaneously for each Henkin measure, see \cite[Chapter 9]{Rudin08}.
We can now generalize this fact. The proof of the following result is modelled after the proof
of \cite[Proposition 5.7]{CD16}.

\begin{prop}
  \label{prop:measure_concentrated}
  Let $\nu \in \TSH$. Then $\nu$ is concentrated on an $F_\sigma$ set that is $\Mult(\cH)$-totally null.
\end{prop}

\begin{proof}
  We may again assume that $\nu$ is a positive probability measure.
  By Corollary \ref{C:singular_fa}, there exists a sequence $(f_n)$ in the unit ball of $A(\cH)$
  that tends to zero in the weak-$*$ topology of $\Mult(\cH)$ and converges to $1$ $\nu$-almost everywhere.
  In other words, if we define
  \begin{equation*}
    E = \{ \zeta \in \partial \bB_d: \lim_{n \to \infty} f_n(\zeta) = 1 \},
  \end{equation*}
  then $E$ is Borel and $\nu(E) = 1$. We claim that $E$ is $\Mult(\cH)$-totally singular.
  To this end, let $\mu$ be a positive $\Mult(\cH)$-Henkin measure that is concentrated
  on $E$. Then by the dominated convergence theorem,
  \begin{equation*}
    \mu(E) =  \lim_{n \to \infty} \int_E f_n \, d \mu = \lim_{n \to \infty} \int_{\partial \bB_d} f_n \, d \mu
    = 0,
  \end{equation*}
  as $\mu$ is $\Mult(\cH)$-Henkin and $f_n$ tends to zero in the weak-$*$ topology of $\Mult(\cH)$.
  Consequently, Lemma \ref{lem:TN_elementary} implies that $E$ is $\Mult(\cH)$-totally null.
  Finally, regularity of $\nu$ implies that $E$ admits an $F_\sigma$ subset $F$ with $\nu(F) = 1$.
  Then $\nu$ is concentrated on the $\Mult(\cH)$-totally null set $F$, as desired.
\end{proof}

\section{The second dual}

We can now obtain a description of the second dual of $A(\cH)$. 
As explained in Subsection \ref{ss:os}, $(\Mult(\cH)_*)^* = \Mult(\cH)$ completely isometrically.
Thus, by Theorem \ref{thm:A_H_decomp}, we have that, completely isometrically,
\begin{align*}
  A(\cH)^{**} &\cong ( \Mult(\cH)_* \oplus_1 \TSH)^* \\&
  =  (\Mult(\cH)_*)^* \oplus_\infty \TSH^* \\&
  = \Mult(\cH) \oplus_\infty \TSH^*.
\end{align*}
It remains to get a better handle on $\TSH^*$.

By Lemma \ref{lem:hen_band}, every measure in $M(\partial \bB_d)$ decomposes uniquely as the $\ell^1$-sum of a 
$\HENH$ measure and a $\TSH$ measure.
Hence
\[ M(\partial \bB_d) = \HENH \oplus_1 \TSH \]
isometrically. Therefore, we obtain the isometric decomposition
\[ M(\partial \bB_d)^* = \HENH^* \oplus_\infty \TSH^* , \]
which is in fact completely isometric since the map from the left to the right is completely contractive and isometric, so it is completely isometric because $M(\partial \bB_d)^* = C(\partial \bB_d)^{**}$ is a commutative
von Neumann algebra; see \cite[1.2.6]{BL04}.
Moreover, we see that 
\begin{align*}
 \TSH^* &= \HENH^\perp \\
 \intertext{and}
  \HENH^* &= \TSH^\perp .
\end{align*}

The double dual of any C*-algebra $\cA$ is $*$-isomorphic and weak-$*$ homeomorphic to the universal enveloping von Neumann algebra $\kW$ of $\cA$, see  \cite[Theorem III.2.4]{TakesakiI02}.
The universal representation $\pi_u$ of $C(\partial \bB_d)$ is equivalent to the direct sum of the standard representations $\pi_\mu$ on $L^2(\mu)$ as $\mu$ runs over all Borel probability measures on $\partial \bB_d$.
By Lemma~\ref{lem:band_decom}, we obtain a decomposition $\pi_u \simeq \pi_{ua} \oplus \pi_{us}$ where $\pi_{ua}$ is $\HENH$ absolutely continuous
and $\pi_{us}$ is $\TSH$ absolutely continuous.
Indeed, $\pi_{ua}$ is equivalent to the direct sum of all $\pi_\mu$ as $\mu$ runs over all $\HENH$ probability measures;
and  $\pi_{us}$ is equivalent to the direct sum of all $\pi_\mu$ as $\mu$ runs over all $\TSH$ probability measures.
Moreover Lemma~\ref{lem:band_decom} shows  that $\pi_{ua}$ and $\pi_{us}$ are mutually orthogonal. This yields a decomposition
\begin{align*}
 \kW &= \ol{\pi_u(C(\partial \bB_d))}^{w^*} \\&
 = \ol{\pi_{ua}(C(\partial \bB_d))}^{w^*} \oplus_\infty \ol{\pi_{us}(C(\partial \bB_d))}^{w^*} \\&
 =: \kW_a \oplus_\infty \kW_s . 
\end{align*}

We claim that $\kW_a = \HENH^*$ and $\kW_s = \TSH^*$.
By Goldstine's theorem (see \cite[Proposition V.4.1]{Conway90} or \cite[II.A.13]{Wojtaszczyk91}), there is a net $(f_\alpha)$ in the unit ball of $C(\partial \bB_d)$ that converges to
\[ 0 \oplus I \in C(\partial \bB_d)^{**} = \HENH^* \oplus_\infty \TSH^* . \]
Since $\HENH^* \!=\! \TSH^\bot$ and $\TSH^* \!=\! \HENH^\bot$\!\!, we see that
$(f_\alpha)$ converges to $0$ weak-$*$ in $L^\infty(\mu)$ for every $\mu \in \HENH$
and converges to 1 weak-$*$ in $L^\infty(\nu)$ for all $\nu\in\TSH$.
It follows that in $\ol{\pi_u(C(\partial \bB_d))}^{w^*}$, we have $\pi_{ua}(f_\alpha) \to 0$ and $\pi_{us}(f_\alpha) \to I$ in the weak-$*$ topology.
Thus $\pi_u(f_\alpha) \to 0 \oplus I =:P$ in $\kW = \kW_a \oplus_\infty \kW_s$.
Now if $\mu\in\HENH$, then 
\[ \langle \mu,P \rangle = \lim_\alpha  \langle \mu,f_\alpha \rangle = 0 .\]
If $\mu \ge0$, then $\mu$ annihilates every $0 \le A \le P$ in $\kW$.
It follows that $\kW_s \subset \HENH^\perp$.
Similarly, $\kW_a \subset  \TSH^\perp$.
However as noted above 
\[ \HENH^\perp \cap \TSH^\perp \!=\! \TSH^* \cap \HENH^* \!=\! \{0\} .\]
Therefore
\[ \TSH^* = \HENH^\perp = \kW_s .\]

Putting everything together, we obtain a useful description of the second dual.

\begin{thm}\label{T:second_dual}
Let $\cH$ be a regular unitarily invariant space on $\bB_d$. 
Let $\kW = \kW_a \oplus \kW_s$ be the decomposition of the universal enveloping von Neumann algebra of $C(\partial \bB_d)$ into its Henkin and totally singular summands. 
Then there is a completely isometric isomorphism
\begin{equation*}
  A(\cH)^{**} \cong    \Mult(\cH) \oplus_\infty \kW_s .
\end{equation*}
\end{thm}

\section{Pick and peak interpolation in \texorpdfstring{$A(\cH)$}{A(H)}} \label{S:interpolation}

Before we are able to prove our result regarding Pick and peak interpolation,
we require the following generalization of Tietze's extension theorem.
If $X$ is metrizable (which is sufficient for our purposes), this is a special case of Dugundji's extension of Tietze's theorem \cite[Theorem 4.1]{Dugundji51}.
The existence of the extension can also be deduced from a classical theorem of Borsuk that asserts the existence of a contractive linear operator of extension; 
see for instance \cite[Theorem 4.4.4]{AK06}.
Instead, we argue in a more elementary way.

\begin{lem}
  \label{lem:Tietze}
  Let $X$ be a compact Hausdorff space and let $K \subset X$ be a closed
  subset. Then the restriction map
  \begin{equation*}
    R: C(X) \to C(K), \quad f \mapsto f \big|_K,
  \end{equation*}
  is a complete quotient map. In fact, $R^{(n)}$ maps the closed
  unit ball of $M_n(C(X))$ onto the closed unit ball of $M_n(C(K))$ for all $n \in \mathbb{N}$.
\end{lem}

\begin{proof}
This result is just a complete version of Tietze's extension theorem.
Let $f: K \to M_n$ be continuous with $\|f(x)\| \le 1$ for all $x \in K$.
We show that there exists a continuous extension $F: X \to M_n$ of $f$ that satisfies $\|F(x)\| \le 1$ for all $x \in X$.
Applying Tietze's extension theorem to each matrix entry of $f$, we obtain a continuous extension $G: X \to M_n$ of $f$. 
We now modify $G$ so that it satisfies the norm constraint.
To this end, consider the continuous function
\begin{equation*}
 \alpha: [0,\infty) \to [0,1], \quad t \mapsto
 \begin{cases}
  1, & \text{ if } t \le 1, \\
  t^{-1/2}, & \text{ if } t > 1,
 \end{cases}
\end{equation*}
and define
\begin{equation*}
F: X \to M_n, \quad x \mapsto G(x) \alpha \big( G(x)^* G(x)\big).
\end{equation*}
Standard properties of the continuous functional calculus (see, for instance,
\cite[II.2.3.3]{Blackadar06}) show that $F$ is continuous.
If $x \in K$, then $G(x)^* G(x) \le I$, hence $F(x)= G(x) = f(x)$ for $x \in K$.
Finally, for any $x \in X$, we find that
\begin{equation*}
F(x)^* F(x) = \alpha \big(G(x)^* G(x) \big) G(x)^* G(x) \alpha \big(G(x)^* G(x) \big) = \beta \big(G(x)^* G(x)\big),
\end{equation*}
where $\beta(t) = t \alpha(t)^2$. 
Since $\beta(t) \le 1$ for all $t \ge 0$, we obtain  $F(x)^* F(x) \le I$; whence $\|F(x)\| \le 1$.
\end{proof}

Let $\cH$ be a regular unitarily invariant space on $\bB_d$.
If $F \subset \bB_d$ is a finite subset, we let
\begin{equation*}
  \Mult(\cH) \big|_F = \{ f \big|_F: f \in \Mult(\cH) \}.
\end{equation*}
Then $\Mult(\cH) \big|_F$ is the quotient of $\Mult(\cH)$ by the subspace of all multipliers vanishing on $F$. 
We endow $\Mult(\cH) \big|_F$ with the corresponding quotient norm,
and indeed with the corresponding quotient operator space structure.
Thus if $g \in M_n(\Mult(\cH) \big|_F)$, then
\begin{equation*}
  \|g\|_{M_n(\Mult(\cH)|_F)} = \inf \{ \|f\|_{M_n(\Mult(\cH))} :
  f \big|_F = g \}.
\end{equation*}

We are now ready to prove a result about Pick and peak interpolation
for spaces on the ball.

\begin{thm}
  \label{thm:pick_peak_general}
  Let $\cH$ be a regular unitarily invariant space on $\bB_d$. Let $E \subset \partial \bB_d$
  be compact and $\Mult(\cH)$-totally null and let $F \subset \bB_d$ be finite. Then
  \begin{equation*}
    \Phi: A(\cH) \to \Mult(\cH) \big|_F \oplus_\infty C(E),
    \quad f \mapsto (f \big|_F, f \big|_E),
  \end{equation*}
  is a complete quotient mapping.
\end{thm}

\begin{proof}
  It suffices to show that the adjoint map
  \begin{equation*}
    \Phi^* : (\Mult(\cH) \big|_F)^* \oplus_1 M(E) \to A(\cH)^*
  \end{equation*}
  is a complete isometry. By Theorem \ref{thm:A_H_decomp},
  $A(\cH)^* = \Mult(\cH)_* \oplus_1 \TSH$ completely isometrically,
  so we may regard $\Phi^*$ as a map
  \begin{equation*}
    (\Mult(\cH) \big|_F)^* \oplus_1 M(E) \to \Mult(\cH)_* \oplus_1 \TSH.
  \end{equation*}
  We will show that $\Phi^*$ respects this direct sum decomposition
  and is completely isometric on each summand. It then follows
  that $\Phi^*$ is a complete isometry. (This can be seen, for instance,
  by noticing that $\Phi^{**}$, being the $\infty$-direct sum
  of two complete quotient maps, is a complete quotient map.)

  First, we show that $\Phi^*$ maps $(\Mult(\cH) \big|_F)^*$ completely isometrically
  into $\Mult(\cH)_* \subset A(\cH)^*$.
  Let $F = \{z_1,\ldots,z_n\}$
  and let $\tau_{z_i} \in (\Mult(\cH) \big|_F)^*$ and
  $\delta_{z_i} \in \Mult(\cH)_*$
  denote the functionals of evaluation at $z_i$ on $\Mult(\cH) \big|_F$
  and on $\Mult(\cH)$, respectively.
  Since $F$ is finite,
  $(\Mult(\cH) \big|_F)^*$ is spanned by the $\tau_{z_i}$, and clearly
  $\Phi^*(\tau_{z_i}) = \delta_{z_i}$.
  In particular, $\Phi^*$ maps $\Mult(\cH \big|_F)^*$ into $\Mult(\cH)_*$.
  By definition,
  \begin{equation*}
    R: \Mult(\cH) \to \Mult(\cH) \big|_F, \quad f \mapsto f \big|_F,
  \end{equation*}
  is a complete quotient mapping, hence its adjoint $R^*: (\Mult(\cH) \big|_F)^* \to \Mult(\cH)^*$
  is a complete isometry. Since $R^*(\tau_{z_i}) = \delta_{z_i}$,
  it follows that $\Phi^*$ is completely isometric on $\Mult(\cH \big|_F)^*$.

  We finish the proof by showing that $\Phi^*$ maps $M(E)$ completely isometrically
  into $\TSH$.
  Let $\mu \in M(E)$. We may trivially extend $\mu$ to a measure on $\partial \bB_d$,
  which we continue to denote by $\mu$.
  Then $\Phi^*(\mu) \in A(\cH)^*$ is simply the integration functional $\rho_{\mu}$.
  Since $E$ is $\Mult(\cH)$-totally null, $\mu \in \TSH$. Thus, $\Phi^*$ maps $M(E)$
  into $\TSH$.
  To show that $\Phi^*$ is a complete isometry on $M(E)$, it suffices to observe that
  the inclusion $M(E) \subset M(\partial \bB_d)$ is completely isometric,
  which follows from Lemma \ref{lem:Tietze} and duality.
\end{proof}

The Pick property makes it possible to explicitly compute the norm in $\Mult(\cH) \big|_F$
with the help of Pick matrices. In this setting, we therefore obtain a more concrete
version of the last result. In particular, Theorem \ref{thm:pick_peak_intro}
is the special case $n=1$ in the following corollary.

\begin{cor}
  \label{cor:pick_peak}
  Let $\cH$ be a regular unitarily invariant space on $\bB_d$ with kernel $K$
  that has the $M_n$-Pick property.
  Let $z_1,\ldots,z_k \in \bB_d$, $W_1,\ldots,W_k \in M_n$ with
  \begin{equation*}
    \big[ K(z_i,z_j) ( I_n - W_i W_j^*) \big] \ge 0.
  \end{equation*}
  Let $E \subset \partial \bB_d$ be compact and $\Mult(\cH)$-totally null and let $h \in M_n(C(E))$
  with $\|h\|_\infty \le 1$.
  Then for each $\varepsilon > 0$, there exists $f \in M_n(A(\cH))$ with
  \begin{enumerate}[label=\normalfont{(\arabic*)}]
    \item $f(z_i) = W_i$ for $1 \le i \le k$,
    \item $f |_E = h$, and
    \item $\|f\|_{M_n(\Mult(\cH))} \le 1 + \varepsilon$.
  \end{enumerate}
\end{cor}

\begin{proof}
The $M_n$-Pick property implies that there exists $g \in M_n(\Mult(\cH))$
with $\|g\|_{M_n(\Mult(\cH))} \le 1$ so that $g(z_i) = W_i$ for $1 \le i \le k$.
Then Theorem \ref{thm:pick_peak_general}, applied with $F = \{z_1,\ldots,z_k\}$, yields $f \in M_n(A(\cH))$ with 
$f \big|_E = h$, $f (z_i) = g(z_i) = W_i$ for $1 \le i \le k$ and $\|f\|_{M_n(\Mult(\cH))} \le 1 + \varepsilon$.
\end{proof}

\begin{rem}
  As explained in the introduction, one cannot eliminate the $\ep$ in this theorem and the corollary in general.

For the ball algebra, it follows from a theorem of Bishop \cite{Bishop62} that with a totally null set $E \subset \partial \bB_d$ alone,
one can interpolate any function $h\in C(E)$ with a function $f$ in the ball algebra $A(\partial \bB_d)$ of the same norm
and even make $|f(z)| < \|h\|_\infty$ for $z\in \ol{\bB_d}\setminus E$, provided that $h$ is not identically zero.
In \cite{CD16}, in the case of the Drury--Arveson space,
the authors were unable to get $\|f\|_{\Mult(H^2_d)} = \|h\|_\infty$,
but were able to arrange that $|f(z)| < \|h\|$ for $z\in \ol{\bB_d}\setminus E$.
They asked whether one can remove the $\ep$. This question will be resolved positively in Section~\ref{S:peak interpolation}.
\end{rem}

The condition that the set $E \subset \partial \bB_d$ in Theorem \ref{thm:pick_peak_general}
and Corollary \ref{cor:pick_peak} be $\Mult(\cH)$-totally null is necessary.

\begin{prop} \label{P:PPI implies TN}
  Let $\cH$ be a regular unitarily invariant space on $\bB_d$. Let $E \subset \partial \bB_d$
  be a compact set with the property that for every finite set $F \subset \bB_d$,
  the map
  \begin{equation*}
    A(\cH) \to \Mult(\cH) \big|_F \oplus_\infty C(E)
    \quad f \mapsto (f \big|_F, f \big|_E),
  \end{equation*}
  is a quotient mapping. Then $E$ is $\Mult(\cH)$-totally null.
\end{prop}

\begin{proof}
  Let $(F_n)_{n=0}^\infty$ be an increasing sequence of finite sets whose union is dense in $\bB_d$.
  By assumption, there exists for each $n \in \bN$ a function $f_n \in A(\cH)$ with
  \begin{equation*}
    f_n \big|_{F_n} = 0, \quad f_n \big|_E = 1 - \tfrac{1}{n} \quad \text { and }\quad \|f_n\|_{\Mult(\cH)} \le 1.
  \end{equation*}
Then $(f_n)$ is a bounded sequence in $\Mult(\cH)$ that tends
to zero pointwise on a dense subset of $\bB_d$, from which it follows that $(f_n)$
tends to zero in the weak-$*$ topology of $\Mult(\cH)$. Clearly, $(f_n)$ tends to $1$ pointwise on $E$.

To show that $E$ is $\Mult(\cH)$-totally null, let $\mu$ be a positive $\Mult(\cH)$-Henkin measure
that is concentrated on $E$. Then by the dominated convergence theorem,
\begin{equation*}
  \mu(E) = \lim_{n \to \infty} \int_{E} f_n \, d \mu =
  \lim_{n \to \infty} \int_{\partial \bB_d} f_n \, d \mu = 0.
\end{equation*}
Therefore, Lemma \ref{lem:TN_elementary} implies that $E$ is $\Mult(\cH)$-totally null.
\end{proof}

In the case when $\mathcal{H}$ admits non-empty $\Mult(\mathcal{H})$-totally null sets, we will
establish a significant strengthening of the preceding result in Theorem \ref{T:interpolation sets}.

We can also prove a version of Theorem \ref{thm:pick_peak_general} in which the finite set $F \subset \bB_d$
is replaced with an interpolating sequence.
Recall that a sequence $(z_n)$ in $\bB_d$
is said to be interpolating for $\Mult(\cH)$ if the map
\begin{equation*}
  \Phi: \Mult(\cH) \to \ell^\infty, \quad f \mapsto (f(z_n)),
\end{equation*}
is surjective.
Interpolating sequences in complete Pick spaces were characterized by Aleman, M\textsuperscript{c}Carthy, Richter and the second
author in \cite{AHM+17}, which also contains more background on this topic.
The following result generalizes \cite[Theorem 5.12]{CD18}, with
a somewhat simpler proof.

\begin{thm}
  \label{thm:IS_TN}
  Let $\cH$ be a regular unitarily invariant space on $\bB_d$. Let $K \subset \ol{\bB_d}$
  be a compact set satisfying
  \begin{enumerate}[label=\normalfont{(\arabic*)}]
    \item $K \cap \bB_d$ is an interpolating sequence for $\Mult(\cH)$, and
    \item $K \cap \partial \bB_d$ is $\Mult(\cH)$-totally null.
  \end{enumerate}
  Then the restriction map $R: A(\cH) \to C(K)$ is surjective.
\end{thm}

\begin{proof}
  By duality, it suffices to show that the adjoint map $R^*: M(K) \to A(\cH)^*$ is bounded below.
  Set $\Lambda = K \cap \bB_d$ and $E = K \cap \partial \bB_d$.
  There is an isometric isomorphism
  \begin{equation*}
    M(K) \to M(\Lambda) \oplus_1 M(E), \quad \mu \mapsto (\mu \big|_\Lambda, \mu \big|_E),
  \end{equation*}
  where $\mu \big|_A(B) = \mu(A \cap B)$ for Borel sets $A,B \subset \ol{\bB_d}$.
  On the other hand, Theorem \ref{thm:A_H_decomp} shows that $A(\cH)^* = \Mult(\cH)_* \oplus_1 \TSH$,
  hence we may regard $R^*$ as a map
  \begin{equation*}
    R^*: M(\Lambda) \oplus_1 M(E) \to \Mult(\cH)_* \oplus_1 \TSH.
  \end{equation*}
  We will show that $R^*$ respects the direct sum and is bounded below on each summand.

  Since $E$ is $\Mult(\cH)$-totally null by assumption, we see exactly as in the last paragraph of the proof of Theorem \ref{thm:pick_peak_general}
  that $R^*$ maps $M(E)$ isometrically into $\TSH$. Since $\Lambda = (z_n)_{n=0}^\infty$ is a sequence, we may identify
  $M(\Lambda) = \ell^1(\Lambda) = \ell^1(\bN)$, and under this identification, $R^*$ maps $e_n$ to $\delta_{z_n}$, the character
  of evaluation at $z_n \in \bB_d$. Hence, $R^*$ maps $M(\Lambda)$ into $\Mult(\cH)_*$.
  Finally, to see that $R^*$ is bounded below on $M(\Lambda)$, we use that $\Lambda$ is an interpolating sequence.
  This means that the map
  \begin{equation*}
    \Phi: \Mult(\cH) \to \ell^\infty, \quad f \mapsto (f(z_n)),
  \end{equation*}
  is surjective, hence the adjoint $\Phi^*$ is bounded below.
  Moreover, $\Phi$ is weak-$*$--weak-$*$ continuous, hence $\Phi^*$ maps $\ell^1(\bN)$ into $\Mult(\cH)_*$.
  Observe that $\Phi^*(e_n) = \delta_{z_n}$ for all $n \in \bN$, that is, $R^*$ agrees with $\Phi^*$ on $M(\Lambda)$. In particular, $R^*$ is bounded below,
  as desired.
\end{proof}

\begin{rem}
  If $(z_n)$ is an interpolating sequence with corresponding surjection $\Phi: \Mult(\cH) \to \ell^\infty$,
  then the norm of the inverse of the induced isomorphism $\Mult(\cH) / \ker(\Phi) \to \ell^\infty$ is called
  the interpolation constant $\gamma$. Since $\Phi$ is contractive, $\gamma \ge 1$. For $H^2_d$, it was observed in \cite{CD18} that in
  the setting of Theorem \ref{thm:IS_TN}, the norm of the inverse of the induced isomorphism $A(\cH) / \ker(R) \to C(K)$
  is at most $2 \gamma + 1$. The above proof shows that this norm is in fact equal to $\gamma$.
\end{rem}

\begin{rem}
  Lemma \ref{lem:completely_surjective} below shows that $R$ is in fact completely surjective.
\end{rem}

\section{Ideals}
\label{S:ideals}

In \cite{CD18}, a detailed analysis was made of the ideals in $A(H^2_d)$ and their zero sets.
Here we will establish a few of these results which require additional work.

The first thing one needs is an analogue of the F.\ and M.\ Riesz Theorem. 
This was established for Drury-Arveson space in \cite[Theorem 4.7]{CD16}, but the proof was quite different.

\begin{prop}\label{P:FMRiesz}
Let $\cH$ be a regular unitarily invariant space on $\bB_d$.
Let $\cJ$ be a closed ideal in $A(\cH)$, and let $\psi$ be a linear functional in $\cJ^\perp$.
Decompose $\psi = \psi_a + \psi_s$ where $\psi_a \in \Mult(\cH)_*$ and $\psi_s \in \TS(\Mult(\cH))$.
Then $\psi_a$ and $\psi_s$ both annihilate $\cJ$.
\end{prop}

\begin{proof}
  Observe that $A(\cH)^{**} \subset B(\cH)^{**}$ is an algebra. Moreover,
  since multiplication in $A(\cH)^{**}$ is separately weak-$*$ continuous,
  $\cJ^{**} = \cJ^{\perp\perp}$ is an ideal in $A(\cH)^{**}$.
By Theorem~\ref{T:second_dual}, $A(\cH)^{**} \cong    \Mult(\cH) \oplus_\infty \kW_s$.
In particular, the element $0\oplus I$ lies in $A(\cH)^{**}$.
It follows that 
\[ \cJ^{\bot \bot} = \cJ^{\bot \bot} (I \oplus 0) \oplus \cJ^{\bot \bot} (0\oplus I) =: \cJ_a \oplus_\infty \cJ_s \]
decomposes as a direct sum.
Let $\psi \in \cJ^\bot$, and suppose that $\Lambda = \Lambda_a + \Lambda_s$ belongs to $\cJ^{\bot \bot}$. 
Then $\Lambda_a,\Lambda_s \in \cJ^{\bot \bot}$, so $\Lambda(\psi_a) = \Lambda_a(\psi) = 0$.
Thus, $\psi_a \in \cJ^\bot$.
Hence also $\psi_s = \psi - \psi_a \in \mathcal{J}^\bot$.
\end{proof}

\begin{rem}
  With notation as in the last proof, it follows that if $\psi\in A(\cH)^*$, then 
 \[
  \psi_a(f) = \hat\psi((I \oplus 0)f) \qand \psi_s(f) = \hat\psi((0\oplus I)f) ,
 \]
 where $\hat\psi$ is the canonical extension of $\psi$ to an element of $A(\cH)^{***}$.
\end{rem}

Suppose that $\cJ$ is an ideal in $A(\cH)$. Define the zero set 
\[ Z(\cJ) :=\{ z \in \ol{\bB_d} : f(z) = 0 \tfa f\in\cJ \} .\]
Conversely, if $E$ is a closed subset of $\ol{\bB_d}$, let 
\[ \cI(E) = \{ f\in A(\cH) : f|_E = 0 \} .\]
For $E \subset \partial \bB_d$, we also write
\begin{equation*}
  \TS(E) = \{ \mu \in \TSH : \mu \text{ is supported on } E \} .
\end{equation*}
Set $\tilde\cJ$ to be the weak-$*$ closure of $\cJ$ in $\Mult(\cH)$.

The following is a consequence of the previous result.

\begin{cor}\label{C:Jperp}
Let $\cJ$ be a closed ideal in $A(\cH)$, and let $E = Z(\cJ) \cap \partial\bB_d$. Then
\[ \cJ^\perp = \tilde\cJ_\perp \oplus_1 \TS(E) .\]
\end{cor}

\begin{proof}
  It is clear that $\cJ^\bot \supset \tilde \cJ_{\bot} \oplus_1 \TS(E)$.
  Conversely, let $\psi \in \cJ^\bot$ have a decomposition
  \begin{equation*}
    \psi = \varphi + \nu \in \Mult(\cH)_* \oplus_1 \TSH.
  \end{equation*}
  Then $\varphi,\nu \in \cJ^\bot$ by Proposition \ref{P:FMRiesz}.
  Since $\varphi$ extends to a weak-$*$ continuous functional on $\Mult(\cH)$, it belongs to $\tilde \cJ_\bot$.

  It remains to show that $\nu$ is supported on $E$.
If $f\in\cJ$, then $f\nu$ annihilates $A(\cH)$; and so is the zero functional.
Since $f\nu \ll \nu$, this measure is totally singular by Lemma~\ref{lem:hen_band}.
Thus by Theorem~\ref{thm:A_H_decomp}, $f\nu = 0$.
That means that the support of $\nu$ is contained in $Z(f)$.
Since this is true for all $f \in \cJ$, it follows that the support of $\nu$ is contained in $E$,
as desired.
\end{proof}

This allows us to obtain an analogue of the Rudin-Carleson theorem describing the ideals of the disk algebra \cite{Rudin57, Carleson57}.
In the case of $A(\bD)$, we know that the weak-$*$ closed ideals of $H^\infty(\bD)$ are of the form $\omega H^\infty(\bD)$ for all inner functions $\omega$,
which yields a precise description of ideals of $A(\bD)$.
The somewhat less precise analogues were established for the ball algebra by Hedenmalm \cite{Hedenmalm89} 
and for Drury-Arveson space in \cite[Theorem 4.1]{CD18}.

\begin{thm}\label{T:ideals}
Let $\cH$ be a regular unitarily invariant space on $\bB_d$, and let $\cJ$ be a closed ideal in $A(\cH)$.
Let $E = Z(\cJ) \cap \partial\bB_d$, and let $\tilde\cJ$ be the weak-$*$ closure of $\cJ$ in $\Mult(\cH)$.
Then
\[ \cJ = \tilde\cJ \cap \cI(E) .\]
\end{thm}

\begin{proof}
  Clearly, $\cJ \subset \tilde \cJ \cap \cI(E)$. Conversely, by
Corollary~\ref{C:Jperp},
  \begin{equation*}
    \cJ^\bot = \tilde \cJ_\bot \oplus_1 \TS(E) \subset (\tilde \cJ \cap \cI(E))^\bot,
  \end{equation*}
  hence $\cJ = \tilde\cJ \cap \cI(E)$ by the Hahn--Banach theorem.
\end{proof}

The following result about ideals of $A(\cH)$ can be established in exactly the same manner as in \cite{CD18}.
We provide a different argument.

\begin{thm}\label{T:I(E)dense}
Let $\cH$ be a regular unitarily invariant space on $\bB_d$. 
Suppose that $E$ is a closed subset of $\partial\bB_d$.
Then $\cI(E)$ is weak-$*$ dense in $\Mult(\cH)$ if and only if $E$ is $\Mult(\cH)$-totally null. 
In this case, the unit ball of $\cI(E)$ is weak-$*$ dense in the unit ball of $\Mult(\cH)$.
\end{thm}

\begin{proof}
 Assume first that $\cI(E)$ is weak-$*$ dense in $\Mult(\cH)$ and let $\mu$ be a positive
 $\Mult(\cH)$-Henkin measure supported on $E$. The integration functional $\rho_\mu$
 on $A(\cH)$ annihilates $\cI(E)$ and extends to a weak-$*$ continuous functional on $\Mult(\cH)$. 
 Hence by weak-$*$ density of $\cI(E)$, $\rho_\mu$ is the zero functional. In particular,
 $\mu(E) = \rho_\mu(1) = 0$, so that $E$ is $\Mult(\cH)$-totally null by Lemma \ref{lem:TN_elementary}.

 Conversely, suppose that $E$ is $\Mult(\cH)$-totally null, and let $f$ belong to the open
 unit ball of $\Mult(\cH)$. Theorem \ref{thm:pick_peak_general} implies that for every finite
 set $F \subset \bB_d$, there exists $g$ in the unit ball of $A(\cH)$ with $g|_E = 0$ and
 $g |_F = f |_F$; whence $g\in \cI(E)$. Since the weak-$*$ topology on the unit ball of $\Mult(\cH)$ agrees
 with the topology of pointwise convergence on $\bB_d$, we see that the unit ball of $\cI(E)$
 is weak-$*$ dense in the unit ball of $\Mult(\cH)$.
\end{proof}

Combining Theorem~\ref{T:I(E)dense} with Corollary~\ref{C:Jperp} immediately yields the following.

\begin{cor}\label{C:I(E)perp}
Let $\cH$ be a regular unitarily invariant space on $\bB_d$. 
Suppose that $E$ is a closed $\Mult(\cH)$-totally null subset of $\partial\bB_d$.
Then $\cI(E)^\perp = \TS(E)$. \qed
\end{cor}

\section{Peak interpolation} \label{S:peak interpolation}

Our results of Section~\ref{S:interpolation} show that a closed $\Mult(\cH)$-totally null set is an interpolation set for $A(\cH)$.
In this section, we will show that it is in fact a peak interpolation set.
This in turn implies that it is a zero set.
Our proof is modelled on a proof of the Rudin-Carleson Theorem from  \cite[III.E.2]{Wojtaszczyk91}.
Not only is this argument simpler than the proof due to Bishop \cite{Bishop62} (see \cite[section 10.3]{Rudin08}), 
it provides sharp control of the multiplier norm. 
This is stronger than the result for Drury-Arveson space in \cite{CD16}, and also considerably easier.

Recall that an $M$-ideal $\cJ$ in a Banach space $\cX$ is a subspace such that $\cX^*$ decomposes as $\cX^* = \cJ^\perp \oplus_1 \cZ$.
These subspaces of $\cX^*$ are called $L$-summands, and there is the identification $\cZ \cong \cJ^*$.
Generalizing a result for C*-algebras, Effros and Ruan \cite{ER90} show that the $M$-ideals of a (approximately) unital operator algebra $\cA$
are precisely the closed two-sided ideals with a contractive approximate unit.
From this, it is immediate that an $M$-ideal $\cJ$ is a complete $M$-ideal, meaning that $M_n(\cJ)$ is an $M$-ideal in $M_n(\cA)$ for all $n\ge1$
(see \cite[Theorem 4.8.5]{BL04}).
One important elementary property of an $M$-ideal $\cJ$ is that it is proximinal, meaning that if $a \in \cA$, there is an element $j \in \cJ$ so that 
$\|a-j\| = \dist(a,\cJ)$ (see \cite[III.D.4]{Wojtaszczyk91} or \cite[Proposition II.2.1]{HWW93}).
See \cite{HWW93} for more background on this topic.

\begin{thm}  \label{thm:Bishop}
Let $\cH$ be a regular unitarily invariant space on $\bB_d$, and let $E \subset \partial \bB_d$
be compact and $\Mult(\cH)$-totally null. Let $g \in M_n(C(E))$ be not identically zero. 
Then there exists $f \in M_n(A(\cH))$ such that
 \begin{enumerate}[label=\normalfont{(\arabic*)}]
   \item $f \big|_E = g$,
   \item $\|f(z)\| < \|g\|_\infty$ for every $z \in \ol{\bB_d} \setminus E$, and
   \item $\|f\|_{M_n(\Mult(\cH))} = \|g\|_\infty$.
 \end{enumerate}
\end{thm}

\begin{proof}
Consider the restriction map
\begin{equation*}
R: A(\cH) \to C(E), \quad f \mapsto f \big|_E.
\end{equation*}
By Theorem~\ref{thm:pick_peak_general}, $R$ is a complete quotient map with kernel $\cI(E)$. 
By Corollary~\ref{C:I(E)perp}, $\cI(E)^\perp = \TS(E)$.
Alternatively, since $R$ has closed range, $\cI(E)^\bot = \ran(R^*) = \TS(E)$.
Observe that by Theorem~\ref{thm:A_H_decomp}, we have
\begin{align*}
 A(\cH)^* &= \Mult(\cH)_* \oplus_1 \TSH \\
 &= \big( \Mult(\cH)_* \oplus_1 \TS(\partial\bB_d\setminus E) \big) \oplus_1 \TS(E) .
\end{align*}
The second decomposition merely splits a totally singular measure $\nu$ as $\nu = \nu_{E^c} + \nu_E$ where
$\nu_F(X) = \nu(X \cap F)$. It is evident that this is a decomposition into $L$-summands.
It follows that $\cI(E)$ is a complete $M$-ideal.

Since $M$-ideals are proximinal, it then follows that $R^{(n)}$ maps the closed unit ball of $M_n(A(\cH))$ onto the closed unit ball of $M_n(C(E))$.
In particular, given $g$ as in the statement of the theorem,
there exists $F \in M_n(A(\cH))$ with $F \big|_E = g$ and $\|F\|_{M_n(A(\cH))} \le \|g\|_\infty$.
Since the multiplier norm dominates the supremum norm also for matrices, it follows that $\|F(z)\| \le \|g\|_\infty$
for all $z \in \ol{\bB_d}$.
To obtain the strict pointwise inequality, it suffices to construct $h$ in the closed unit ball of $A(\cH)$
with $h \big|_E = 1$ and $|h(z)| < 1$ for $z \in \partial \bB_d \setminus E$. 
It then follows from the maximum modulus principle that $|h(z)| < 1$ for all $z \in \ol{\bB_d} \setminus E$.
Therefore $f = h F$ satisfies all requirements.

We now construct $h$.
Notice that since $E$ is a non-empty $\Mult(\cH)$-totally null set, the singleton $\{z\}$ is $\Mult(\cH)$-totally null for every $z \in E$. 
Unitary invariance of $\cH$ then implies that all singleton sets in $\partial\bB_d$ are $\Mult(\cH)$-totally null.
For each $z \in \partial \bB_d \setminus E$, the union $E \cup \{z\}$ is totally null.
By the previous paragraph, there is a function $h_z$ in the closed unit ball of $A(\cH)$ with $h_z \big|_E = 1$ and $h_z(z) = 0$. 
By continuity of $h_z$, there exists an open neighborhood $U_z$ of $z$ in $\partial \bB_d$ so that $|h_z(w)| < \frac{1}{2}$ for all $w \in U_z$.
Since $\partial \bB_d \setminus E$ is second countable, the open cover $(U_z)_{z \in \partial \bB_d \setminus E}$ has a countable subcover $(U_{z_k})_{k=1}^\infty$. 
Set
\begin{equation*}
h = \sum_{k=1}^\infty 2^{-k} h_{z_k}.
\end{equation*}
Then $h$ belongs to the closed unit ball of $A(\cH)$, $h \big|_E = 1$ and $|h(z)| < 1$ for $z \in \partial \bB_d \setminus E$, as desired.
\end{proof}

This readily yields the fact that every closed $\Mult(\cH)$-totally null set is a zero set.

\begin{cor} \label{C:zero sets}
  Let $E \subset \partial \bB_d$ be compact and $\Mult(\cH)$-totally null. 
  Then there exists $f \in A(\cH)$ with $E = \{z \in \ol{\bB_d} : f(z) = 0 \}$.
\end{cor}

\begin{proof}
  By Theorem \ref{thm:Bishop}, there exists $g \in A(\cH)$ with $g \big|_E = 1$
  and $|g(z)| < 1$ for all $z \in \ol{\bB_d} \setminus E$. Set $f = 1 - g$.
\end{proof}

Using more sophisticated results from the theory of $M$-ideals, we can even obtain a linear operator
of peak interpolation, i.e.\ we can achieve that in Theorem \ref{thm:Bishop}, the function $f$ depends
linearly on $g$. In the case of the disc algebra and the ball algebra, this was first proved
by Pe\l czy\'{n}ski \cite{Pelczynski64}.

It is a theorem of And\^o \cite{Ando73a} and Choi--Effros \cite{CE77} that if $\cJ \subset \cX$ is an $M$-ideal so that
$\cX/\cJ$ is isometrically isomorphic to $C(K)$ for a compact metric space $K$, then
there exists a linear contractive lifting $L: \cX / \cJ \to \cX$, i.e.\ $L$ is a right-inverse
of the quotient mapping. More generally, a contractive lifting exists whenever $\cJ$ is an $M$-ideal so
that $\cX / \cJ$ is a separable Banach space that satisfies the metric approximation property.
This result can be found in \cite[Theorem II.2.1]{HWW93}; see also \cite[Section II.6]{HWW93}
for a discussion of the history and special cases of this result. From this, one
obtains a linear operator of peak interpolation in the scalar case. To deal with the matrix case,
we will use an operator space version of the lifting theorem due to Effros and Ruan \cite{ER94}.

\begin{thm}
  \label{thm:Bishop_linear}
Let $\cH$ be a regular unitarily invariant space on $\bB_d$ and let $E \subset \partial \bB_d$
be compact and $\Mult(\cH)$-totally null.
Then there exists a linear complete isometry $L: C(E) \to A(\cH)$ so that
 \begin{enumerate}[label=\normalfont{(\arabic*)}]
   \item $L(g) \big|_E = g$ for all $g \in C(E)$, and
   \item $\|L^{(n)}(g)(z)\| < \|g\|_\infty$ for all $g \in M_n(C(E)) \setminus \{0\}$ and every $z \in \ol{\bB_d} \setminus E$.
 \end{enumerate}
\end{thm}

\begin{proof}
  We first construct a linear complete contraction $L_0: C(E) \to A(\cH)$ that satisfies (1). Recall
  from the first paragraph of the proof of Theorem \ref{thm:Bishop} that the restriction map
  \begin{equation*}
    R: A(\cH) \to C(E), \quad f \mapsto f \big|_E,
  \end{equation*}
  is a complete quotient mapping whose kernel is a complete $M$-ideal.
  In this setting, \cite[Theorem 5.2]{ER94} implies that there exists a linear complete contraction $L_0: C(E) \to A(\cH)$
  so that $R \circ L_0$ is the identity, i.e.\ (1) holds.
 
 To see that \cite[Theorem 5.2]{ER94} is applicable, one has to observe that $C(E)$ satisfies the operator
  metric approximation property, and that $A(\cH)$ is locally reflexive. The first assertion is
  a standard partition of unity argument; see, \cite[Proposition 2.4.2]{BO08}. The second
  assertion follows from the fact that $C^*(A(\cH))$, thanks to the short exact sequence
  \eqref{eqn:short_exact} in Subsection \ref{ss:RKHS}, is a nuclear $C^*$-algebra \cite[Exercise 3.8.1]{BO08}.
  Hence it locally reflexive (see \cite[Proposition 18.19]{Pisier03} or \cite[Corollary 9.4.1]{BO08}), and local reflexivity passes to subspaces \cite[p. 185]{ER94}.

  To obtain property (2), we apply Theorem \ref{thm:Bishop} to get a function $h$ in the closed unit ball of $A(\cH)$
  with $h \big|_E = 1$ and $|h(z)| < 1$ for $z \in \ol{\bB_d} \setminus E$. Define
  \begin{equation*}
    L: C(E) \to A(\cH), \quad g \mapsto h L_0(g).
  \end{equation*}
  Then $L$ is a linear complete contraction that satisfies (1) and (2), using once
  more that the multiplier norm dominates the supremum norm. From (1), we deduce that $L$ is in fact
  a complete isometry, which finishes the proof.
\end{proof}

\section{Existence of totally null sets}

In this section, we study when a regular unitarily invariant space admits non-empty $\Mult(\cH)$-totally null
sets. We begin with an easy lemma.

\begin{lem}
  \label{lem:TN_existence}
  Let $\cH$ be a regular unitarily invariant space. The following
  assertions are equivalent:
  \begin{enumerate}[label=\normalfont{(\roman*)}]
    \item There exists a non-empty $\Mult(\cH)$-totally null set.
    \item For each $z \in \partial \bB_d$, the singleton set $\{z\}$ is
      $\Mult(\cH)$-totally null.
    \item For each $z \in \partial \bB_d$, the Dirac measure $\delta_z$ is $\Mult(\cH)$-totally singular.
    \item $\HENH \subsetneq M(\partial \bB_d)$.
  \end{enumerate}
\end{lem}

\begin{proof}
  (i) $\Rightarrow$ (ii) Let $E$ be a non-empty $\Mult(\cH)$-totally null set. Then
  for each $z \in E$, the singleton $\{z\}$ is $\Mult(\cH)$-totally null. Unitary invariance
  of $\cH$ then implies that all singleton subsets of $\partial \bB_d$ are $\Mult(\cH)$-totally null.

  (ii) $\Rightarrow$ (iii) Since the Dirac measure $\delta_z$ is concentrated on $\{z\}$,
  it is $\Mult(\cH)$-totally singular by Lemma \ref{lem:TN_elementary}.

  (iii) $\Rightarrow$ (iv) is trivial.

  (iv) $\Rightarrow$ (i) By Lemma \ref{lem:hen_band}, there exists a non-zero positive $\Mult(\cH)$-totally singular
      measure $\nu$. Proposition \ref{prop:measure_concentrated} implies that $\nu$ is concentrated
      on a $\Mult(\cH)$-totally null set $E$. Since $\nu \neq 0$, $E \neq \emptyset$, so (i) holds.
\end{proof}

In Remark \ref{rem:everything_henkin}, we observed that if the kernel $K(z,w) = \sum_{n=0}^\infty a_n \langle z,w \rangle^n$  of a regular unitarily invariant space satisfies $\sum_{n=0}^\infty a_n < \infty$, i.e.\ if $K$ is bounded,
then $\HENH = M(\partial \bB_d)$ and hence there are no non-empty
$\Mult(\cH)$-totally null sets. In the presence of the $2$-point Pick property, the converse holds as well.

\begin{prop}
  \label{prop:unbounded_Pick}
  Let $\cH$ be a regular unitarily invariant space on $\bB_d$ satisfying the $2$-point Pick property.
  If the reproducing kernel of $\cH$ is unbounded, then
  each singleton subset of $\partial \bB_d$ is $\Mult(\cH)$-totally null.
\end{prop}

\begin{proof}
  By Lemma \ref{lem:TN_existence}, it suffices to show that the Dirac measure $\delta_{e_1}$ is
  not $\Mult(\cH)$-Henkin. The $2$-point Pick property of $\cH$ allows us to solve, for each $0 < r < 1$, the extremal
  problem
  \begin{equation*}
    \sup \{ \Re( f(r e_1)) : f \in \Mult(\cH), \|f\|_{\Mult(\cH)} \le 1, f(0) = 0 \}.
  \end{equation*}
  The unique multiplier that achieves the supremum is given by
  \begin{equation*}
    f_r(z) = \frac{1 - \frac{1}{K(z,r e_1)}}{\sqrt{1 - \frac{1}{K(r e_1, r e_1)}}},
  \end{equation*}
  see \cite[Proposition 3.1]{Hartz17a}. Notice that $f_r$ is holomorphic in a neighborhood
  of $\ol{\bB_d}$, hence $f_r \in A(\cH)$ for all $0 <r < 1$; see the discussion in Subsection \ref{ss:RKHS}.
  Since $\sum_{n=0}^\infty a_n = \infty$, we see that $\lim_{r \to 1} K(r e_1, r e_1) = \infty$,
  so $f_r$ converges pointwise on $\bB_d$ to
  \begin{equation*}
    f(z) = 1 - \frac{1}{K(z,e_1)}.
  \end{equation*}
  Note that  $K(z,e_1)$ is defined since $| \langle z, e_1 \rangle| < 1$ even though $e_1$ is not in the open ball.
  In fact, since $\|f_r\|_{\Mult(\cH)} \le 1$ for all $0 < r <1$, it follows that
  $f_r$ converges to $f$ in the weak-$*$ topology of $\Mult(\cH)$,
  and therefore $\|f\|_{\Mult(\cH)} \le 1$.

  Suppose now for a contradiction that the Dirac measure $\delta_{e_1}$ is $\Mult(\cH)$-Henkin,
  and let $\varphi \in \Mult(\cH)_*$ be the unique weak-$*$ continuous extension of the
  functional of evaluation at $e_1$. Then $\varphi$ is multiplicative.
  Moreover,
  \begin{equation*}
    \varphi(f) = \lim_{r \to 1} \varphi(f_r) = \lim_{r \to 1} f_r(1) = 1,
  \end{equation*}
  hence $\varphi(f^n) = 1$ for all $n \in \bN$.
  On the other hand, $|f(z)| < 1$  for all $z \in \bB_d$ by the maximum modulus principle.
 Hence $(f^n)$ tends to zero in the weak-$*$ topology of $\Mult(\cH)$, so
  $\varphi(f^n)$ tends to zero, a contradiction.
\end{proof}

The following basic observation is sometimes useful to show
that singleton sets are totally null.

\begin{lem}
  \label{lem:Henkin_inclusion}
  Let $\cH,\cK$ be two regular unitarily invariant spaces on $\bB_d$ with
  $\Mult(\cH) \subset \Mult(\cK)$. Then
  \begin{enumerate}[label=\normalfont{(\alph*)}]
    \item $\HEN(\Mult(\cK)) \subset \HEN(\Mult(\cH))$, and
    \item each $\Mult(\cH)$-totally null set is $\Mult(\cK)$-totally null.
  \end{enumerate}
\end{lem}

\begin{proof}
  (a) The closed graph theorem implies that the inclusion $\Mult(\cH) \subset \Mult(\cK)$ is bounded.
  Hence Lemma \ref{lem:Henkin_sequence} implies that every $\Mult(\cK)$-Henkin measure is also $\Mult(\cH)$-Henkin.

  (b) This is immediate from (a).
\end{proof}

\begin{exa}
  Consider the spaces $\cH_s(\bB_d)$, defined in Subsection \ref{ss:RKHS}.
  We claim that singleton sets are $\Mult(\cH_s)$-totally null if and only if $s \ge -1$.

  To see this, note that for $s < -1$,
  the coefficients $(a_n)$ in the kernel are summable, hence there are no non-empty
  totally null sets. If $s \in [-1,0]$, then $\cH_s(\bB_d)$
  is a complete Pick space by a straightforward extension of \cite[Corollary 7.41]{AM02}. Hence Proposition \ref{prop:unbounded_Pick} shows that
  singleton sets are totally null if $s \in [-1,0]$.
  Finally, it is well known that $\Mult(\cH_s(\bB_d)) \subset \Mult(\cH_t(\bB_d))$
  for $s \le t$.
  Indeed, this is a special case of \cite[Corollary 3.4]{AHM+18a}. In the present setting, it is also
  an elementary consequence of the fact that for $s > -1$, the space $\cH_s(\bB_d)$ admits an equivalent norm for which the 
  reproducing kernel is given by
  \begin{equation*}
    K_s(z,w) = \frac{1}{(1 - \langle z,w \rangle)^{s+1}}.
  \end{equation*}
  Since $K_t / K_s$ is positive semi-definite if $s \le t$, it follows that $\Mult(\cH_s(\bB_d)) \subset \Mult(\cH_t(\bB_d))$
  for $s \le t$, see for instance \cite[Corollary 3.5]{Hartz17}.
  Lemma \ref{lem:Henkin_inclusion} therefore implies that singleton sets are $\Mult(\cH)$-totally null if $s \ge -1$.
\end{exa}

Next, we will construct an example to show that Proposition \ref{prop:unbounded_Pick} may fail without
the Pick property. This will be accomplished with the help of the following lemma.
Recall from the discussion in Subsection \ref{ss:RKHS} that if $\cH$ is a regular unitarily invariant
space on $\bD$, then the spectrum of $M_z$ on $\cH$ is $\ol{\bD}$. In particular,
the spectral radius formula implies that $\lim_{n \to \infty} \|z^n\|_{\Mult(\cH)}^{1/n} = 1$.
This is essentially the only restriction on the rate of growth of $\|z^n\|_{\Mult(\cH)}$, even
if we insist that the kernel be unbounded.

\begin{lem}
  \label{lem:kernel_big_mult_norm}
  Let $(\alpha_n)_{n=1}^\infty$ be a sequence in $[1,\infty)$ satisfying $\dlim_{n \to \infty} \alpha_n^{1/n} = 1$.
  Then there exists a regular unitarily invariant space $\cH$ on $\bD$ 
  with unbounded kernel and
  $\|z^n\|_{\Mult(\cH)} \ge \alpha_n$ for all $n \ge 1$.
\end{lem}

\begin{proof}
  Let $S$ be the unilateral weighted shift whose weight sequence $(w_n)$ is given by
  \begin{equation*}
    (\alpha_1,1,\ldots,1,\alpha_2^{1/2},\alpha_2^{1/2},1,\ldots,1,\alpha_3^{1/3},\alpha_3^{1/3},\alpha_3^{1/3},1,\ldots),
  \end{equation*}
  where the block of ones following $k$ copies of $\alpha_k^{1/k}$ consists of $n_k$ ones, 
  and we require that $n_k \ge (\alpha_1 \alpha_2 \ldots \alpha_k)^{2}$.
  Then $S$ is unitarily equivalent to $M_z$ on the reproducing kernel Hilbert space $\cH$ with reproducing kernel
  \begin{equation*}
    K(z,w) = \sum_{n=0}^\infty a_n (z \overline{w})^n,
  \end{equation*}
  where $a_0 = 1$, and $(w_n)$ and $(a_n)$ are related by
  \begin{equation}
    \label{eqn:weights_shift}
    \begin{split}
    w_n^2 &= \frac{a_n}{a_{n+1}} \quad (n \ge 0), \\
    a_n &= \prod_{k=0}^{n-1} w_k^{-2} \quad (n \ge 1),
    \end{split}
  \end{equation}
  see \cite[Section 3]{Shields74}. The assumption on $(\alpha_n)$ implies that
  \begin{equation*}
    \lim_{n \to \infty} \frac{a_n}{a_{n+1}} = \lim_{n \to \infty} w_n^2 = 1,
  \end{equation*}
  so $\cH$ is regular. Moreover,
  \begin{equation*}
    \|z^n\|_{\Mult(\cH)} = \|S^n\| = \sup_{k \ge 0} w_k w_{k+1} \ldots w_{k+n-1} \ge \alpha_n
  \end{equation*}
  for all $n \ge 1$. Finally, to see that $\sum_{n=0}^\infty a_n = \infty$, notice that
  the relation \eqref{eqn:weights_shift} implies that for each $k \ge 1$, the sequence $(a_n)$
  equals $(\alpha_1 \alpha_2 \ldots \alpha_k)^{-2}$ at least $n_k$ times, so $\sum_{n=0}^\infty a_n$ diverges by choice of $n_k$.
\end{proof}

This lemma allows us to construct spaces whose multipliers are very regular close to the boundary of $\bD$,
which in turn implies that there are no non-empty totally null sets.

 \begin{prop}
  \label{prop:unbounded_henkin}
  There exists a regular unitarily invariant space $\cH$ on $\bD$ with unbounded kernel
  that admits no non-empty $\Mult(\cH)$-totally null sets.
\end{prop}

\begin{proof}
  By Lemma \ref{lem:kernel_big_mult_norm}, there exists a regular unitarily invariant space $\cH$ on $\bD$
  with unbounded kernel so that $\|z^n\|_{\Mult(\cH)} \ge (n+1)^3$ for all $n \ge 0$.
  We claim that if $f \in \Mult(\cH)$, then $f \in A(\bD)$ and $f' \in A(\bD)$.
  To this end, notice that if $f \in \Mult(\cH)$ has a Taylor series $f(z) = \sum_{n=0}^\infty \widehat{f}(n) z^n$, then
  \begin{equation*}
    \widehat{f}(n) z^n = \frac{1}{2 \pi} \int_{0}^{2 \pi} f(e^{ it} z) e^{- i nt} \, dt .
  \end{equation*}
  This integral actually converges in the weak-$*$ topology by Lemma \ref{lem:multiplier_rotations}.
  Thus unitary invariance of $\cH$ implies that
  \begin{equation*}
    \|\widehat{f}(n) z^n \|_{\Mult(\cH)} \le \|f\|_{\Mult(\cH)} .
  \end{equation*}
  Therefore
  \begin{equation*}
    |\widehat{f}(n)| \le \frac{\|f\|_{\Mult(\cH)}}{(n+1)^3}
  \end{equation*}
  for all $n \in \bN$. It follows that the Taylor series of $f$ and of $f'$ converge absolutely in $A(\bD)$. Moreover,
  there exists a constant $C > 0$ so that
  \begin{equation*}
    \|f'\|_\infty \le C \|f\|_{\Mult(\cH)} 
  \end{equation*}
  for all $f \in \Mult(\cH)$.

  We finish the proof by showing that for each $z \in \partial \bD$, the Dirac measure $\delta_z$ is $\Mult(\cH)$-Henkin.
  For $w \in \ol{\bD}$, let $\tau_w \in A(\cH)^*$ denote the functional of evaluation at $w$.
  If $r \in (0,1)$ and $z \in \partial \bD$, then for each $f \in A(\cH)$,
  \begin{equation*}
    |f(r z) - f(z)| \le (1 -r) \|f'\|_\infty \le (1 -r) C \|f\|_{\Mult(\cH)} .
  \end{equation*}
  Hence $\tau_{r z}$ converges to $\tau_z$ in the norm of $A(\cH)^*$. Since $\tau_{r z}$ is $\Mult(\cH)$-Henkin
  and since $\Mult(\cH)$-Henkin functionals form a norm closed subspace of $A(\cH)^*$, we conclude that $\delta_{z}$ is a $\Mult(\cH)$-Henkin
  measure, as desired.
\end{proof}

The existence of non-empty $\Mult(\cH)$-totally null sets leads to a dichotomy in
the boundary behavior of multipliers.

\begin{prop}
  \label{prop:Henkin_dichotomy}
  Let $\cH$ be a regular unitarily invariant space on $\bB_d$.
  \begin{enumerate}[label=\normalfont{(\alph*)}]
    \item If singleton sets are not $\Mult(\cH)$-totally null,
      then $\Mult(\cH) \subset A(\bB_d)$, and evaluation at each point
      in $\ol{\bB_d}$ is weak-$*$ continuous on $\Mult(\cH)$.
    \item If singleton sets are $\Mult(\cH)$-totally null,
      then every sequence $(z_n)$ in $\bB_d$ with $\lim_{n \to \infty} \|z_n\| = 1$
      has a subsequence that is interpolating for $\Mult(\cH)$.
  \end{enumerate}
\end{prop}

\begin{proof}
  (a) If singleton sets are not $\Mult(\cH)$-totally null,
  then Lemma \ref{lem:TN_existence} implies that for each $\zeta \in \partial \bB_d$,
  the Dirac measure $\delta_\zeta$ is $\Mult(\cH)$-Henkin.
  Thus, for each $\zeta \in \partial \bB_d$, there exists a unique weak-$*$ continuous functional
  $\tau_\zeta$ on $\Mult(\cH)$ that agrees with evaluation at $\zeta$ on $A(\cH)$.
  We also write $\tau_z$ for the weak-$*$ continuous functional
  of evaluation at $z \in \bB_d$ on $\Mult(\cH)$.
  Let $f \in \Mult(\cH)$ and define
  \begin{equation*}
    g: \ol{\bB_d} \to \bC, \quad
    z \mapsto \tau_z(f) .
  \end{equation*}
  Clearly $g$ agrees with $f$ on $\bB_d$.
  We claim that $g$ is continuous.
  It suffices to show that $g$ is continuous at each $\zeta \in \partial \bB_d$.

  Let $f_{r,U}(z) = f(r U z)$ for $0 \le r \le 1$, $U \in \cU_d$ and $z \in \bB_d$.
  We claim that
  \begin{equation}
    \label{eqn:rotation_character}
    \tau_{r U \zeta}(f) = \tau_\zeta(f_{r,U}).
  \end{equation}
  Indeed, part (a) of Lemma \ref{lem:multiplier_rotations} shows that the map
  \begin{equation*}
    \Mult(\cH) \to \Mult(\cH), \quad f \mapsto f_{r,U},
  \end{equation*}
  is weak-$*$--weak-$*$ continuous and maps $A(\cH)$ into $A(\cH)$.
  So \eqref{eqn:rotation_character} holds for all $f \in A(\cH)$ by definition of $\tau_\zeta$.
  Both sides of the equation are weak-$*$ continuous in $f$, so it holds for all $f \in \Mult(\cH)$.

  To show that $g$ is continuous at $\zeta \in \partial \bB_d$,
  let $(z_n)$ be a sequence in $\ol{\bB_d}$ that converges to $z$.
  An elementary linear algebra argument shows that there exist a sequence $(r_n)$ in $[0,1]$
  tending to $1$ and a sequence $(U_n)$ in $\cU_d$ tending to $I$ so that $z_n = r_n U_n \zeta$ for all $n$.
  Part (b) of Lemma \ref{lem:multiplier_rotations} implies that $f_{r_n,U_n}$
  tends to $f$ in the weak-$*$ topology of $\Mult(\cH)$.
  Hence using \eqref{eqn:rotation_character} and weak-$*$ continuity of $\tau_\zeta$, we find that
  \begin{equation*}
    g(z_n) = \tau_{z_n}(f)
    = \tau_{\zeta}(f_{r_n,U_n}) \xrightarrow{n \to \infty} \tau_\zeta(f) = g(\zeta).
  \end{equation*}
  Thus, $g$ is continuous at every point of $\partial \bB_d$, and hence is continuous on $\ol{\bB_d}$.

  It is clear from the definition of the extension $g$ of $f$ that evaluation at each point in $\ol{\bB_d}$
  is weak-$*$ continuous.

  (b)
  By passing to a subsequence, we may assume that $(z_n)$ tends to $\zeta \in \partial \bB_d$.
  Theorem \ref{thm:Bishop} shows that there exists a function
  $f \in A(\cH)$ that peaks at $\zeta$, i.e. $f(\zeta) = 1$, $|f(\zeta)| < 1$ for all
  $z \in \ol{\bB_d} \setminus \{\zeta\}$ and $\|f\|_{\Mult(\cH)} = 1$.
  In this setting, a result of Douglas and Eschmeier
  (see Lemma 12 and the discussion preceding Corollary 14 in \cite{DE12})
  shows $(z_n)$ admits a subsequence that is interpolating for $\Mult(\cH)$.
  Explicitly, let $w_n = f(z_n)$. Then $(w_n)$ is a sequence in $\mathbb{D}$ that tends to $1$,
  hence by passing to a subsequence, we may achieve that $(w_n)$ is interpolating for $H^\infty(\mathbb{D})$.
  By testing on kernel functions, one checks that $(M_f^*)^n$ tends to zero in the strong
  operator topology, so the contraction $M_f$ admits an $H^\infty$-functional calculus.
  This means that $h \circ f \in \Mult(\mathcal{H})$ for all $h \in H^\infty$,
  from which it follows that $(z_n)$ is interpolating for $\Mult(\mathcal{H})$.
\end{proof}

\begin{rem}
  Proposition \ref{prop:Henkin_dichotomy} shows that in the proof of Proposition \ref{prop:unbounded_henkin},
  it would have been sufficient to arrange that $\Mult(\cH) \subset A(\bD)$. But the proof Proposition \ref{prop:unbounded_henkin}
  is more elementary as it does not rely on the $H^\infty$-functional calculus
  or on Theorem \ref{thm:Bishop}.
\end{rem}

The following consequence about existence of interpolating sequences is immediate from Proposition \ref{prop:Henkin_dichotomy}.
 
\begin{cor}
  Let $\cH$ be a regular unitarily invariant space on $\bB_d$ and suppose that $\Mult(\cH) \not \subset A(\bB_d)$.
  Then there exist infinite interpolating sequences for $\Mult(\cH)$. \qed
\end{cor}

\section{Interpolation sets} \label{S:interpolation sets}

Let $\cH$ be a regular unitarily invariant space.
We say that a compact set $E \subset \partial \bB_d$ is an \emph{interpolation set} for $A(\cH)$
if the restriction map $A(\cH) \to C(E)$ is surjective.
Theorem \ref{thm:Bishop}, or already Theorem \ref{thm:pick_peak_general},
implies in particular that every $\Mult(\mathcal{H})$-totally null set is an interpolation
set for $A(\mathcal{H})$. We will establish the converse under the assumption that there exist non-empty $\Mult(\mathcal{H})$-totally
null sets at all.

Recall that a linear map $\Phi: X \to Y$ between operator spaces is said to be
\emph{completely surjective} if there exists a constant $C > 0$ so that for all $n \in \bN$
and for all $y \in M_n(Y)$, there exists $x \in M_n(X)$ with $\Phi^{(n)}(x) = y$
and $\|x\| \le C \|y\|$.

\begin{lem}
  \label{lem:completely_surjective}
  Let $\cA$ be a complete unital operator algebra, let $X$ be a compact Hausdorff space and let $\Phi: \cA \to C(X)$
  be a unital homomorphism.
  \begin{enumerate}[label=\normalfont{(\alph*)}]
    \item If $\Phi$ is surjective, then $\Phi$ is completely surjective.
    \item If $\Phi$ is a quotient map, then $\Phi$ is a complete quotient map.
  \end{enumerate}
\end{lem}

\begin{proof}
  Since characters on unital Banach algebras are contractive, $\Phi$ is contractive, and since $C(X)$
  is a commutative $C^*$-algebra, $\Phi$ is in fact completely contractive.
  In each case, $\Phi$ is surjective, so the open mapping theorem
  implies that the induced map
  \begin{equation*}
    \widetilde{\Phi}: \cA / \ker(\Phi) \to C(X)
  \end{equation*}
  has a bounded inverse $\Psi$, which clearly is a unital homomorphism.
  Using again that $C(X)$ is a commutative $C^*$-algebra, we find that $\Psi$
  is completely bounded with $\|\Psi\|_{cb} \le \|\Psi\|^2$; see \cite[Theorem 9.7]{Paulsen02}.
  This proves (a). Moreover, if $\Phi$ is a quotient map,
  then $\|\Psi\| \le 1$, hence $\|\Psi\|_{cb} \le 1$, which proves (b).
\end{proof}

\begin{rem}
  This provides another proof of the first part of Lemma \ref{lem:Tietze}.
 It is a bit shorter, but relies on more machinery about completely bounded maps.
\end{rem}  

We are now ready to show that interpolation sets are totally null.
Our proof is inspired by the proof in \cite[Theorem 10.2.2]{Rudin08} of a theorem of Varopoulos.
The main difference is that we replace pointwise arguments involving a clever choice of roots of unity
with arguments involving matrices of multipliers.

\begin{thm} \label{T:interpolation sets}
  Let $\cH$ be a regular unitarily invariant space and let $E \subset \partial \bB_d$ be compact.
  Suppose that there exist non-empty $\Mult(\cH)$-totally null sets.
  If $E$ is an interpolation set for $A(\cH)$, then $E$ is $\Mult(\cH)$-totally null.
\end{thm}

\begin{proof}
Let $\mu$ be a positive $\Mult(\cH)$-Henkin measure concentrated on $E$. We have to show that $\mu(E) = 0$;
see Lemma \ref{lem:TN_elementary}.
Since $E$ is an interpolation set, an application of Lemma \ref{lem:completely_surjective} yields
that the restriction map $A(\cH) \to C(E)$ is completely surjective, say with constant $C > 0$.

In the first step, we will show that for each $\varepsilon > 0$ and for each compact set $K \subset \bB_d$,
there exists $f \in A(\cH)$ with
\begin{enumerate}[label=\normalfont{(\alph*)}]
\item $\mu( \{ |1 - f| \ge \varepsilon \}) < \varepsilon$,
\item $|f| \le \varepsilon C^2$ on $K$, and
\item $\|f\|_{\Mult(\cH)} \le C^2$.
\end{enumerate}

Since there exist non-empty $\Mult(\cH)$-totally null sets, each singleton set is totally
null by unitary invariance of $\cH$. Hence Theorem \ref{thm:Bishop} implies that for
each $z \in E$, there exists $f_z \in A(\cH)$ with $f_z(z) = 1$, $|f_z(w)| < 1$ for $w \in \ol{\bB_d} \setminus \{z\}$ and $\|f_z\|_{\Mult(\cH)} = 1$. By replacing $f_z$ with a sufficiently large power of $f_z$,
we may achieve in addition that $|f_z| \le \varepsilon$ on $K$.
Compactness of $E$ allows us to find finitely many $f_i = f_{z_i}$, $1 \le i \le n$, so that
the open sets $\{ |1 - f_i| < \varepsilon \}$ cover $E$. By regularity of $\mu$, there exist
disjoint compact sets $E_1,\ldots,E_n$ so that
\begin{equation}
\label{eqn:E_i_inclusion}
E_i \subset E \cap \{ |1 - f_i| < \varepsilon \}
\end{equation}
for each $i$
and
\begin{equation}
\label{eqn:measure_complement}
\mu \Big( \partial \bB_d \setminus \bigcup_{i=1}^n E_i \Big) < \varepsilon,
\end{equation}
as $\mu$ is concentrated on $E$.
Since the restriction map $A(\cH) \to C(E)$ is completely surjective with constant $C$,
Tietze's extension theorem (Lemma~\ref{lem:Tietze}) implies
that the restriction map $A(\cH) \to C( \cup_i E_i)$ is also completely surjective with constant $C$. 
Therefore,  since the $E_i$ are disjoint, there exist
$g_1,\ldots,g_n \in A(\cH)$ so that
$g_i = \delta_{i j}$ on $E_j$ for $1 \le i,j \le n$, and
\begin{equation*}
\|
 \begin{bmatrix}
  g_1 & \cdots & g_n
 \end{bmatrix} \|_{\Mult(\cH \otimes \bC^n, \cH)} \le C.
\end{equation*}
By the same token, there exist
$h_1,\ldots,h_n \in A(\cH)$ so that
$h_i = \delta_{i j}$ on $E_j$ for $1 \le i,j \le n$, and
\begin{equation*}
 \Bigg\|
 \begin{bmatrix}
   h_1 \\ \vdots \\ h_n
 \end{bmatrix} \Bigg\|_{\Mult(\cH, \cH \otimes \bC^n)} \le C.
\end{equation*}
Let
\begin{equation*}
 f = \sum_{i=1}^n g_i f_i h_i
 = 
 \begin{bmatrix}
  g_1 & \cdots & g_n
 \end{bmatrix}
 \begin{bmatrix}
  f_1 & 0 & \cdots & 0 \\
  0 & f_2 & \cdots & 0 \\
  \vdots & \ddots  & \ddots & \vdots \\
  0 & 0  & \cdots & f_n
 \end{bmatrix}
 \begin{bmatrix}
  h_1 \\ \vdots \\ h_n
 \end{bmatrix}.
\end{equation*}
Since $\|f_i\|_{\Mult(\cH)} \le 1$ for each $i$, the matrix representation
shows that $\|f\|_{\Mult(\cH)} \le C^2$, i.e.\ (c) holds.
If $z \in K$, then $|f_i(z)| \le \varepsilon$ for each $i$, hence using
that the supremum norm is dominated by the multiplier norm also for matrices,
we find that $|f(z)| \le\varepsilon C^2$, i.e.\ (b) holds.
To show (a), notice that if $z \in E_i$, then $f(z) = f_i(z)$,
hence $|1 - f(z)| < \varepsilon$ by \eqref{eqn:E_i_inclusion}. In conjunction with
\eqref{eqn:measure_complement}, this yields that $\mu \{ |1 - f| \ge \varepsilon \} < \varepsilon$,
i.e.\ (a) holds. This finishes the construction of $f$.

Applying the first step to a sequence $(\varepsilon_n)$ decreasing to zero and to the compact
sets $K = r_n \ol{\bB_d}$, where $r_n \in (0,1)$ increases to $1$, we obtain
a bounded sequence $(f_n)$ in $A(\cH)$ that converges to $0$ uniformly on compact subsets of $\bB_d$
and to $1$ in $\mu$-measure. Thus, $(f_n)$ tends to zero in the weak-$*$ topology of $\Mult(\cH)$.
From the dominated convergence theorem and the fact that $\mu$ is $\Mult(\cH)$-Henkin, we infer that
\begin{equation*}
\mu (E) = \lim_{n \to \infty} \int_{\partial \bB_d} f_n \, d\mu
= 0.
\end{equation*}
This shows that $E$ is $\Mult(\cH)$-totally null, as desired.
\end{proof}

If there do not exist non-empty $\Mult(\cH)$-totally null sets, then there
are no non-trivial interpolation sets for $A(\cH)$.

\begin{prop}
  \label{prop:no_interpolation}
  Let $\cH$ be a regular unitarily invariant space on $\bB_d$ that does not admit
  non-empty $\Mult(\cH)$-totally null sets. Then a compact set $E \subset \partial \bB_d$
  is an interpolation set for $A(\cH)$ if and only if $E$ is finite.
\end{prop}

\begin{proof}
  Clearly, finite sets are interpolation sets since $A(\cH)$ contains the polynomials.

  Conversely, assume for a contradiction
  that $E$ is an infinite interpolation set for $A(\cH)$.
  Then $E$ contains a sequence $(z_n)$ of distinct
  points that converges to some $z \in E$.
  Applying the open mapping theorem to the restriction map $A(\cH) \to C(E)$ and using
  Tietze's extension theorem,
  we obtain a constant $C > 0$ so that for each $n \in \bN$, there exists a function
  $f_n \in A(\cH)$ with $f_n(z_k) = (-1)^k$ for $k \le n$ and $\|f_n\|_{A(\cH)} \le C$.

  Weak-$*$ compactness of the unit ball of $\Mult(\cH)$ shows that the sequence $(f_n)$
  has a weak-$*$ cluster point $f \in \Mult(\cH)$. Since $\cH$ does not admit non-empty $\Mult(\cH)$-totally
  null sets, part (a) of Proposition \ref{prop:Henkin_dichotomy}
  shows that every function in $\Mult(\cH)$ has a unique extension to a continuous
  function on $\ol{\bB_d}$ and that evaluation at each point in $\partial \bB_d$ is weak-$*$ continuous.
  It follows that $f(z_k) = (-1)^k$ for all $k \in \bN$.
  This contradicts the fact that $\lim_{k \to \infty} z_k = z$ and that $f$ is continuous at $z$.
\end{proof}

\section{Zero sets} \label{S:zero_sets}

It is known that a compact set
$E \subset \partial \bB_d$ is $H^\infty(\bB_d)$-totally null if and only if it is the
zero set of a function in the ball algebra; see \cite[Chapter 10]{Rudin08}.
In the Drury--Arveson space, a theorem of Clou\^atre and the first author \cite[Proposition 5.1]{CD18}
shows that every compact $\Mult(H^2_d)$-totally null subset of $\partial \bB_d$ is a zero set of a function in $A(H^2_d)$. 
We have generalized this by establishing that compact $\Mult(\cH)$-totally null sets are zero sets for $A(\cH)$ in Corollary~\ref{C:zero sets}.
The authors of \cite{CD18} also ask in \cite[Questions 5.2 and 5.3]{CD18} whether the converse holds.
We take this opportunity to point out that the example constructed in \cite{Hartz17} provides a negative answer to this question.

\begin{thm}
\label{thm:zero}
For each $d \ge 2$, there exists a function $f \in A(H^2_d)$ whose
zero set $\{z \in \ol{\bB_d}: f(z) = 0\}$ is contained in $\partial \bB_d$
and supports a $\Mult(H^2_d)$-Henkin probability measure.
In particular, the zero set of $f$ is not $\Mult(H^2_d)$-totally null.
\end{thm}

In \cite{Hartz17}, a probability measure $\mu$ on $\partial \bB_d$ was constructed
such that $\mu$ is $\Mult(H^2_d)$-Henkin, but the support $E$ of $\mu$ is $H^\infty(\bB_d)$-totally
null. We will show that $E$ is the zero set of a function in $A(H^2_d)$.

It is easy to see that if $f \in A(H^2_d)$ satisfies the conclusion of Theorem \ref{thm:zero},
then for any $d' \ge d$, the trivial extension $f \circ P$ of $f$ to $\bB_{d'}$, where $P$ is the orthogonal
projection onto the first $d$ coordinates, satisfies the conclusion
of the theorem as well. To see this, one has to observe that the map
\begin{equation*}
  A(H^2_d) \to A(H^2_{d'}), \quad f \mapsto f \circ P,
\end{equation*}
is continuous (in fact a complete isometry), which follows from an explicit computation
with the reproducing kernels or from the von Neumann inequality for $H^2_{d'}$. Moreover,
one has to check that the trivial extension of a $\Mult(H^2_d)$-Henkin
measure to $\partial \bB_{d'}$ is $\Mult(H^2_{d'})$-Henkin,
see \cite[Lemma 2.3]{Hartz17}.

It therefore suffices to establish Theorem \ref{thm:zero} for $d = 2$. In fact,
the construction in \cite{Hartz17} was significantly easier for $d=4$, so we will consider
that case first.

\begin{proof}[Proof of Theorem \ref{thm:zero} for $d=4$]
 We use the construction in Section 3 of \cite{Hartz17}.
 Let
\begin{equation*}
  r: \ol{\bB_4} \to \ol{\bD}, \quad z \mapsto 16 z_1 z_2 z_3 z_4,
\end{equation*}
and let $E = r^{-1}(1)$. The arithmetic mean--geometric mean inequality
implies that $r$ maps $\ol{\bB_4}$ onto $\ol{\bD}$ and $\bB_4$ onto $\bD$.
In particular, $E \subset \partial \bB_4$.
In \cite[Section 3]{Hartz17}, a $\Mult(H^2_4)$-Henkin measure $\mu$ supported on $E$ was
constructed as follows. Let
\begin{equation*}
  h: \mathbb{T}^3 \to \partial \mathbb{B}_4, \quad (\zeta_1,\zeta_2,\zeta_3) \mapsto \frac{1}{2} (\zeta_1,\zeta_2,\zeta_3,\overline{\zeta_1 \zeta_2 \zeta_3}),
\end{equation*}
let $m$ be the normalized Lebesgue measure on $\mathbb{T}^3$, and let
$\mu$ be the pushforward of $m$ by $h$:
\begin{equation*}
  \mu(A) = m( h^{-1}(A))
\end{equation*}
for Borel subsets $A \subset \partial \mathbb{B}_4$.
Since the range of $h$ is contained in $E$, the measure $\mu$ is supported on $E$.
Moreover, by \cite[Lemma 3.3]{Hartz17}, $\mu$ is $\Mult(H^2_4)$-Henkin.
Let $f = 1-r$. Then the zero set of $f$ is $E$; and since $f$ is a polynomial, $f \in A(H^2_4)$.
\end{proof}

We remark that the use of the arithmetic mean-geometric mean inequality
on the monomials $z_1 z_2 \ldots z_d$ in order to construct counterexamples in the Drury--Arveson space
already appears in Arveson's work, see the proof of \cite[Theorem 3.3]{Arveson98}.

The basic idea behind the proof in the case $d=2$ is the same, but as in \cite{Hartz17},
the construction becomes more involved in dimension $2$.
The following argument will establish Theorem \ref{thm:zero}
in full.

\begin{proof}[Proof of Theorem \ref{thm:zero} for $d=2$]
We use the construction in Section 4 of \cite{Hartz17}.
Let $F \subset \bT$ be the circular middle-thirds Cantor set, let
\begin{equation*}
  r: \ol{\bB_2} \to \ol{\bD}, \quad z \mapsto 2 z_1 z_2,
\end{equation*}
and let $E = r^{-1}(F)$.
Once again, $E \subset \partial \bB_2$ by the arithmetic mean--geometric mean inequality.
In \cite[Section 4]{Hartz17}, a $\Mult(H^2_2)$-Henkin measure supported on $E$ was
constructed as follows. Let $\sigma$ be the Cantor measure on $F$, let
\begin{equation*}
  h: \mathbb{T} \times F \to \partial \mathbb{B}_2, \quad (\zeta_1,\zeta_2) \mapsto \frac{1}{\sqrt{2}} (\zeta_1, \overline{\zeta_1} \zeta_2),
\end{equation*}
and let $\mu$ be the pushforward of $m \times \sigma$ by $h$.
Since the range of $h$ is contained $E$, the measure $\mu$ is supported on $E$.
Moreover, by \cite[Lemma 4.6]{Hartz17} (see also \cite[Remark 4.4]{Hartz17}),
$\mu$ is $\Mult(H^2_2)$-Henkin.

It remains to show that $E$ is the zero set of a function in $A(H^2_2)$.
To this end, notice that the middle-thirds Cantor set $F$ is a Carleson set, meaning
that $|F| = 0$ and $\sum_k |I_k| \log (1/|I_k|) < \infty$, where $I_k$ are the connected
components of $\bT \setminus F$ and $|I|$ denotes the linear Lebesgue measure of $|I|$.
By a theorem of Carleson (see, for instance Theorem 4.4.3 and Exercise 4.4.2 in \cite{EKM+14}),
there exists $h \in A(\bD)$ so that $h'' \in A(\bD)$ and so that
$F = \{\zeta \in \bT: h(\zeta) = 0\}$. Let $f = h \circ r$. 
Then the zero set of $f$ is $E$.

We finish the proof by showing that $f \in A(H^2_2)$.
Let $h(z) = \sum_{n=0}^\infty a_n z^n$ be the Taylor series of $h$,
so that
\begin{equation*}
  f(z) = \sum_{n=0}^\infty a_n r(z)^n.
\end{equation*}
Since $h{''} \in A(\bD)$, there exists a constant $C > 0$
so that $|a_n| \le \frac{C}{(n+1)^2}$ for all $n \in \bN$. Moreover,
since $r^n$ is a homogeneous polynomial for all $n \in \bN$, it follows
for instance from \cite[Proposition 6.4]{Hartz17} and the usual
formula of the norm of a monomial in $H^2_d$ that
\begin{equation*}
  \|r^n\|_{\Mult(H^2_2)}^2 = \| r^n\|_{H^2_2}^2
  = 4^n \frac{(n!)^2}{(2 n)!} \sim \sqrt{\pi} \sqrt{n+1},
\end{equation*}
where the last approximation is a consequence of Stirling's formula. Thus,
\begin{equation*}
  \sum_{n=0}^\infty |a_n| \|r^n\|_{\Mult(H^2_2)}
  \le C' \sum_{n=0}^\infty (n+1)^{-7/4} < \infty
\end{equation*}
for some constant $C' > 0$. This shows that the series defining $f$ converges
absolutely in the Banach algebra $A(H^2_2)$, hence $f \in A(H^2_2)$.
\end{proof}

The same principle, but in an easier fashion, also applies to the Dirichlet space $\cD$. 

\begin{prop}
The Cantor middle-thirds set $E \subset \bT$ is the zero set of a function
in $A(\cD)$, but is not $\Mult(\cD)$-totally null.
\end{prop}

\begin{proof}
As in the proof of Theorem \ref{thm:zero}, there exists $f \in A(\bD)$
with $f^{''} \in A(\bD)$ whose zero set is $E$. One checks
that the Taylor series of $f$ converges absolutely in $\Mult(\cD)$, hence $f \in A(\cD)$.
On the other hand, $E$ has positive logarithmic capacity (see, for instance,
the remark at the end of Section 2.4 of \cite{EKM+14}), hence $E$ is not $\Mult(\cD)$-totally null
by Proposition \ref{prop:TN_cap0}.
\end{proof}

The following is an immediate consequence of the proof of Theorem~\ref{thm:zero} for the case $d\ge4$.

\begin{cor}\label{C:weak peaking}
For each $d \ge 4$, there is a compact subset $E \subset \partial\bB_d$ which is not $\Mult(H^2_d)$-totally null
for which, for any $\ep>0$, there is a function $f\in A(H^2_d)$ such that $f|_E = 1$, $|f(z)| < 1$ for all $z \in \ol{\bB_d}\setminus E$ and $\|f\|_{\Mult(\mathcal{H})} \le 1+\ep$.
\end{cor}

\begin{proof}
  Once again, it suffices to consider the case $d=4$.
Let $r(z) = 16 z_1 z_2 z_3 z_4$ and $E = r^{-1}(1)$ as in the proof of Theorem~\ref{thm:zero}.
As observed there, $\|r\|_\infty = 1$. 
Let $f(z) =\tfrac{n + r(z)}{n+1}$ 
for $n$ sufficiently large.
\end{proof}

Of course, these results and examples raise the question about how to characterize zero sets of $A(\cH)$.
In the case of Drury-Arveson space, we see that they are intermediate between the $\Mult(H^2_d)$-totally null sets
and the classical $H^\infty(\bB_d)$-totally null sets, but not the same as the former. The latter are known to characterize the zero sets of $A(\bB_d)$.

\section{Relations among various interpolation properties}
\label{S:relations}

Let us introduce a bit of terminology to facilitate discussion.
\pagebreak[3]

\begin{defn}
A compact subset $E$ of $\partial\bB_d$ is said to be
\begin{enumerate}
    \item[\normalfont{(PI)}] a \emph{peak interpolation set} if the conclusion of Theorem~\ref{thm:Bishop} holds for scalar valued functions;
\item[\normalfont{(I)}] an \emph{interpolation set} if the restriction map of $A(\cH)$ into $C(E)$ is surjective;
\item[\normalfont{(P)}] a peak set if there exists a function $f\in A(\cH)$ such that $f|_E=1$, $|f(z)|<1$ for every $z \in \ol{\bB_d} \setminus E$, and $\|f\|_{\Mult(\mathcal{H})}=1$;
    \item[\normalfont{(PPI)}]
a \emph{Pick-peak interpolation set} if for every finite set $F \subset \bB_d$, the restriction map 
from $A(\cH)$ to $\Mult(\cH)\big|_F \oplus_\infty C(E)$ is a quotient map.
\end{enumerate}
\end{defn}
We also write (TN) to mean that $E$ is $\Mult(\mathcal{H})$-totally null.

Clearly (PI) implies both (P) and (I); and (PPI) implies (I).
Also (P) implies that $E$ is a zero set, and it implies the weaker notion that there is an $f\in A(\cH)$ such that 
$f|_E=1$, $|f(z)|<1$ for every $z \in \ol{\bB_d} \setminus E$ without the sharp norm control on $f$. 
However the results of the previous section show that a set $E$ with these weaker properties need not be totally null.

We can now summarize some of our main results as follows.

\begin{thm} \label{T:equivalence}
Let $\cH$ be a regular unitarily invariant space on $\bB_d$, and 
let $E \subset \partial \bB_d$  be compact. 
The following are equivalent:
  \begin{enumerate}
    \item[\normalfont{(TN)}] $E$ is $\Mult(\cH)$-totally null.
    \item[\normalfont{(PI)}] $E$ is a peak interpolation set.
    \item[\normalfont{(P)}] $E$ is a peak set.
    \item[\normalfont{(PPI)}] $E$ is a Pick-peak interpolation set.
  \end{enumerate}
Moreover these properties imply the corresponding complete versions of {\normalfont{(PI)}} and {\normalfont{(PPI)}} for matrix values functions.
Furthermore, if there exist non-empty $\Mult(\cH)$ totally null sets, then this is also equivalent to
  \begin{enumerate}
    \item[\normalfont{(I)}] $E$ is an interpolation set.
  \end{enumerate}
\end{thm}

\begin{proof}
Theorem~\ref{thm:Bishop} shows that (TN) implies the complete version of (PI), which trivially implies (P).
Suppose that $E$ is a (P) set, and let $f\in A(\cH)$ satisfy $f|_E=1$, $|f(z)|<1$ for every $z \in \ol{\bB_d} \setminus E$, and $\|f\|_{\Mult(\mathcal{H})}=1$.
Let $\mu$ be a positive Henkin measure concentrated on $E$. 
We show that $\mu(E) = 0$.
The sequence $(f^n)$ is bounded in $A(\cH)$ and converges to $0$ pointwise on $\bB_d$.
Hence it converges to $0$ in the weak-$*$ topology of $\Mult(\cH)$. Thus, 
\begin{equation*}
  \mu(E) =  \int_{\partial \bB_d} f^n \, d\mu \xrightarrow{n \to \infty} 0,
\end{equation*}
as desired. So $E$ is $\Mult(\cH)$-totally null.

Theorem~\ref{thm:pick_peak_general} shows that (TN) implies the complete version of (PPI).
The converse is Proposition~\ref{P:PPI implies TN}.

The implication that (PPI) or (PI) implies (I) is trivial.
Finally, if there are $\Mult(\cH)$ totally null sets,  then Theorem~\ref{T:interpolation sets} shows that (I) implies (TN).
\end{proof}


%
\end{document}